\documentclass[10pt,a4paper,reqno]{amsart}
\usepackage[draft]{optional}
\usepackage[margin=0cm]{geometry}
\usepackage{lipsum}
\usepackage[leqno]{amsmath}
\usepackage{mathrsfs}

\usepackage{amsfonts, amsmath, wasysym, caption,cite}

\usepackage{amssymb,amsthm, paralist}

\usepackage{amsopn}
\usepackage{accents,bbm,dsfont,eucal,esint,paralist,url,verbatim,wasysym}
\usepackage[normalem]{ulem}
\usepackage{bbold}

\usepackage{latexsym,nicefrac}
\usepackage[usenames]{color}
\usepackage{tikz}
\usetikzlibrary{decorations.pathreplacing}
\usepackage{url}

\definecolor{darkgreen}{rgb}{0,0.5,0}
\definecolor{darkred}{rgb}{0.7,0,0}

\usepackage[colorlinks,
citecolor=darkgreen, linkcolor=darkred
]{hyperref}

\usepackage{esint}
\usepackage[normalem]{ulem}

\usepackage{enumitem}
\theoremstyle{plain}
\newtheorem{lemma}{Lemma}[section]

\newtheorem{thm}[lemma]{Theorem}
\newtheorem{prop}[lemma]{Proposition}
\newtheorem{cor}[lemma]{Corollary}
\theoremstyle{definition}

\newtheorem{defn}[lemma]{Definition}

\newtheorem{remark}[lemma]{Remark}

\newtheorem{rmk}[lemma]{Remark}

\numberwithin{equation}{section}

\newcommand{\al}{\alpha}
\newcommand{\be}{\beta}
\newcommand{\ga}{\gamma}
\newcommand{\Ga}{\Gamma}
\newcommand{\de}{\delta}
\newcommand{\dee}{\varrho}
\newcommand{\De}{\Delta}
\newcommand{\om}{\omega}
\newcommand{\Om}{\Omega}

\newcommand{\la}{\lambda}

\newcommand{\rh}{\rho}

\renewcommand{\th}{\theta}

\newcommand{\R}{\ensuremath{{\mathbb R}}}

\newcommand{\weakstarto}{\overset{*}{\rightharpoonup}}

\newcommand{\brmk}{\begin{rmk}}
\newcommand{\ermk}{\end{rmk}}
\newcommand{\partref}[1]{\hbox{(\csname @roman\endcsname{\ref{#1}})}}

\newcommand{\beq}{\begin{equation}}
\newcommand{\eeq}{\end{equation}}
\newcommand{\beqs}{\begin{equation*}}
\newcommand{\eeqs}{\end{equation*}}
\newcommand{\beqa}{\begin{equation}\begin{aligned}}
\newcommand{\eeqa}{\end{aligned}\end{equation}}
\newcommand{\beqas}{\begin{equation*}\begin{aligned}}
\newcommand{\eeqas}{\end{aligned}\end{equation*}}

\newcommand{\half}{\frac{1}{2}}

\newcommand{\eps}{\varepsilon}

\newcommand{\supp}{\text{supp}}

\newcommand{\loc}{{\text{\rm loc}}}

\newcommand{\F}{\mathcal{F}}
\newcommand{\dd}{{\rm d}}

\title[Relativistic Euler equations]{{
Global Entropy Solutions and Newtonian Limit\\ for the Relativistic Euler Equations
}}
\author{Gui-Qiang G. Chen}
\address{Gui-Qiang G. Chen:\, Mathematical Institute,
University of Oxford, Oxford,  OX2 6GG, UK.}
\email{\tt chengq@maths.ox.ac.uk}

\author{Matthew R. I. Schrecker}
\address{Matthew R. I. Schrecker:\, Department of Mathematics, University College London, 25 Gordon St, London, WC1H 0AY, UK.}
\email{\tt m.schrecker@ucl.ac.uk}

\keywords{Relativistic Euler equations, entropy solutions,
Newtonian limit, compactness, solution operator, vanishing viscosity method, entropy kernel,
entropy-flux kernel, fundamental solutions, compensated compactness}

\subjclass[2010]{35L65,35Q75,35L03,35L67,35Q35,35A01,76N01,83A05,35A35}
\date{\today}

\begin{document}
\begin{abstract}
We analyze the relativistic Euler equations of conservation laws of baryon number and momentum
with a general pressure law.
The existence of global-in-time bounded entropy solutions for the system
is established by developing a compensated compactness framework.
The proof relies on a careful analysis of the entropy and entropy-flux functions,
which are represented by the \it fundamental solutions \rm
of the entropy and entropy-flux equations for the relativistic Euler equations.
Based on the careful entropy analysis, we establish the compactness framework for sequences of both exact solutions
and approximate solutions of the relativistic Euler equations.
Then we construct approximate solutions via the vanishing viscosity method and employ our compactness framework
to deduce the global-in-time existence of entropy solutions.
The compactness of the solution operator is also established.
Finally, we apply our techniques to establish the convergence of the Newtonian limit
from the entropy solutions of the relativistic Euler equations to the classical Euler equations.
\end{abstract}
\maketitle

\section{Introduction}
The isentropic relativistic Euler equations of conservation laws of baryon number and momentum
are a natural relativistic extension
of the Euler equations for classical fluid flow ({\it i.e.}, in the setting of Newtonian mechanics).
These equations describe the motion of inviscid fluids
in the Minkowski space-time $(t,x)\in \R_+^2:=\R_+\times \R$ in special relativity, which are given by
\begin{equation} \label{eq:relativistic-Euler}
\begin{cases}
\partial_t\big(\frac{n}{\sqrt{1-u^2/c^2}}\big) + \partial_x\big(\frac{nu}{\sqrt{1-u^2/c^2}}\big) = 0, \\[2.5mm]
\partial_t\big(\frac{(\rho + p/c^2)u}{1-u^2/c^2}\big) + \partial_x\big(\frac{(\rho + p/c^2)u^2}{1-u^2/c^2} + p\big) = 0,
\end{cases}
\end{equation}
where $\rho$ and $p$ represent the mass-energy density and pressure respectively,
$u$ is the particle speed,
$n$ is the proper number density of baryons, and $c$ is the light speed.
Henceforth, we write $\eps:= \frac{1}{c^2}$ for notational convenience.
We close the system by imposing the equation of state of a barotropic gas: $p=p(\rho)$.

The proper number density of baryons is determined by the first law of thermodynamics:
$$
T \dd S = \frac{\dd\rho}{n} - \frac{\rho + \eps p}{n^2}\dd n,
$$
where $T$ is the temperature and $S$ is the entropy per baryon.
In particular, for a barotropic fluid under consideration ({\it i.e.}, $S$ is constant),
$$
\frac{\dd n}{n} = \frac{\dd\rho}{\rho + \eps p},
$$
so that
$$
n = n(\rho) = n_0e^{\int_{0}^{\rho}\frac{\dd s}{s + \eps p(s)}}.
$$
By rescaling the first equation in \eqref{eq:relativistic-Euler} if necessary,
we may assume without loss of generality that $n_0 =1$.
By way of comparison with the classical Euler equations,
we observe that, in the Newtonian limit as the light speed $c\to\infty$ (equivalently, $\eps\to0$),
$n(\rho)$ converges to $\rho$, locally uniformly.

Concerning the pressure, a typical example is: $p(\rho) = \kappa \rho^{\gamma}$,
the case of a polytropic (or gamma-law) gas, with the adiabatic exponent $\gamma \in (1,3)$
and  constant $\kappa = \frac{(\gamma - 1)^2}{4\gamma}$.
In this paper, we deal with a more general class of pressure laws, whose explicit conditions
will be given later in \eqref{ass:stricthyperbolicity}--\eqref{ass:genuine-nonlinearity} and \eqref{ass:pressure}.

We focus on the Cauchy problem:
\begin{equation}\label{Cauchy-problem}
(\rho, u)|_{t=0}=(\rho_0(x), u_0(x)) \qquad \mbox{for $x\in \R$}.
\end{equation}
Our approach to the relativistic Euler equations is motivated by the successful strategies
employed in resolving the Cauchy problem for the isentropic Euler equations in the classical setting.
To motivate this comparison, we observe that, formally, in the Newtonian limit ($c\rightarrow \infty$),
system \eqref{eq:relativistic-Euler} reduces to the classical isentropic Euler equations
for compressible fluids:
\begin{equation}\label{eq:classical-Euler}
\begin{cases}
\partial_t\rho + \partial_x(\rho u) = 0,\\[1mm]
\partial_t(\rho u) +  \partial_x(\rho u^2 + p(\rho)) =0.
\end{cases}
\end{equation}
This formal observation raises the question whether the limit here can be taken rigorously:
Do the global entropy solutions of the relativistic Euler equations \eqref{eq:relativistic-Euler}
converge to an entropy solution of the classical Euler equations \eqref{eq:classical-Euler} as $c\to\infty$?
One of the main contributions of this paper is to give an affirmative answer
to this question for the general class of pressure laws including the polytropic case.

To place the relativistic Euler equations in the
general framework of hyperbolic systems of conservation laws,
we introduce some additional notation.
Denote
\beqa\label{eq:conservedvariables}
U = (\frac{n}{\sqrt{1 - \eps u^2}},\frac{(\rho + \eps p(\rho))u}{1 - \eps u^2})^\top,\quad\,\,
F(U) = (\frac{nu}{\sqrt{1 - \eps u^2}},\frac{\rho u^2 + p(\rho)}{1 - \eps u^2})^\top.
\eeqa
Then system \eqref{eq:relativistic-Euler} takes the form:
$$
\partial_t U + \partial_xF(U) =0.
$$
We assume throughout the conditions of strict hyperbolicity:
\beq\label{ass:stricthyperbolicity}
p'(\rho)>0\qquad \text{ for $\rho>0$,}
\eeq
and genuine nonlinearity:
\beq\label{ass:genuine-nonlinearity}
\rho p''(\rho)+2p'(\rho)>0 \qquad \text{ for $\rho>0$}.
\eeq
We remark that, strictly speaking, the condition of genuine nonlinearity for
the relativistic Euler system \eqref{eq:relativistic-Euler}
reads
\beq\label{ass:genuine-nonlinearity-relativistic}
\rho p''(\rho)+2p'(\rho)+\eps\big(p(\rho)p''(\rho)-2p'(\rho)^2\big)>0 \qquad \text{ for $\rho>0$.}
\eeq
In a relativistic fluid, the sound speed is given by the expression:
$$
c_{\rm s}(\rho)=\sqrt{p'(\rho)}.
$$
Thus, to obey the usual laws of relativity, $c_{\rm s}(\rho)$ must always be bounded by the light speed:
$$
c_{\rm s}(\rho)<\frac{1}{\sqrt{\eps}}.
$$
We define $\rho_{\max}^\eps$ such that $c_{\rm s}(\rho_{\max}^\eps)=\frac{1}{\sqrt{\eps}}$
if such a finite $\rho_{\max}^\eps$ exists,
or $\infty$ otherwise.

The relativistic Euler system \eqref{eq:relativistic-Euler} was derived by Taub in \cite{Taub},
in which he also calculated the Rankine-Hugoniot conditions across a shock for the system
and discussed possible pressure laws for relativistic gases.
Further discussion on the pressure-density relation may be found in the work of Thorne \cite{Thorne},
who suggested that the pressure should grow linearly with the density at high densities,
while behaving as a gamma-law gas near the vacuum.

The first global existence result for the relativistic Euler equations was obtained
by Smoller-Temple \cite{SmollerTemple} in the case of an isothermal flow
($\gamma=1$) under the assumption of bounded total variation of the initial data \eqref{Cauchy-problem}.
In this setting,
the Glimm scheme is used to create a convergent sequence of approximate solutions by the random
choice method.
Subsequently, Ding-Li \cite{DingLi,DingLi2} again employed the Glimm scheme to obtain
the global existence of entropy solutions of the relativistic piston problem
for the isentropic Euler equations with the initial data of small total variation,
in which they were also able to show that, in the Newtonian limit,
the relativistic solutions
converge to the entropy solution of the classical piston
problem for the Euler equations.
Liang \cite{Liang} studied the formation of shocks and the structure of simple waves,
based on the work of Taub \cite{Taub}.
The existence of entropy solutions with large data was obtained in Hsu-Lin-Makino \cite{HsuLinMakino}
for a special class of pressure laws under the assumption of sufficiently large speed of light (or equivalently, small data).
Other large data results were obtained by Chen-Li \cite{ChenLi}, showing the existence and stability
of entropy solutions of the Riemann problem for this system,
and the same properties were shown for the variant system of relativistic Euler equations
(system \eqref{eq:relativistic-Euler-alternative} below) in \cite{ChenLi2}.
For system \eqref{eq:relativistic-Euler-alternative}, Li-Feng-Wang  \cite{LiFengWang} were also able
to employ the Glimm scheme to show the existence of entropy solutions for a class of large initial data.
All of these results require
restrictions on the type of pressure laws that can be handled
as well as, for many of them, the conditions on the smallness of total variation.
We significantly weaken these requirements for the existence and compactness of entropy solutions in this paper,
leading to the following theorem, which is our first main result.

\begin{thm}[Existence and Compactness of Entropy Solutions to the Relativistic Euler Equations]\label{thm:relativistic-main}
Let $(\rho_0,u_0)$ be measurable and bounded initial data satisfying
\begin{equation}\label{ID-1}
|u_0(x)|\leq M_0<\frac{1}{\sqrt{\eps}}, \hspace{3mm}  0\leq\rho_0(x)\leq \rho_{M_0}<\rho_{\max}^\eps  \qquad\,\, \text{ for a.e. $x\in\R$,}
\end{equation}
for some constants $M_0>0$ and $\rho_{M_0}$ independent of $\eps$.
Let the pressure function $p(\rho)$ satisfy  \eqref{ass:stricthyperbolicity}--\eqref{ass:genuine-nonlinearity}
for $\rho>0$ and
\begin{equation}\label{ass:pressure}
 p(\rho)=\kappa\rho^\gamma\left(1+P(\rho)\right), \hspace{8mm} |P^{(n)}(\rho)|\leq C\rho^{\gamma-1-n} \text{\,\, for $\,0\leq n\leq 4$},
\end{equation}
for some $\gamma\in(1,3)$.
Then there is $\eps_0>0$ such that, if $\eps\leq\eps_0$,
there exists an entropy solution $(\rho,u)$ of \eqref{eq:relativistic-Euler} $($in the sense
of Definition {\rm \ref{def:relativistic-entropy-solution}} below$)$
such that
$$
|u(t,x)|\leq M<\frac{1}{\sqrt{\eps}}, \hspace{3mm}   0\leq\rho(t,x)\leq \rho_M <\rho_{\max}^\eps  \qquad\,\,
\text{ for a.e. $(t,x)\in\R^2_+:=[0,\infty)\times \R$,}
$$
where the constants $M$ and $\rho_M$ depend only on $M_0$ and $\rho_{M_0}$, independent of $\eps$.
Furthermore, the solution operator determined by the above is compact in $L^r_{\rm loc}(\mathbb{R}^2_+), 1\le r<\infty$, for $t>0$.
\end{thm}

\begin{remark}
Condition \eqref{ass:pressure} can be relaxed to the same condition as in Chen-LeFloch \cite{ChenLeFloch2}.
For brevity, we focus on the class of pressure laws satisfying condition \eqref{ass:pressure} in this paper.
\end{remark}

In addition, our analysis of the relativistic Euler equations is also sufficient to
control the convergence of a sequence of solutions of the relativistic Euler equations as $\eps\to0$,
allowing us to prove our second main theorem.

\begin{thm}[Convergence of the Newtonian Limit]\label{thm:Newtonianlimit}
Let $(\rho_0,u_0)\in (L^\infty(\mathbb{R}))^2$
satisfy \eqref{ID-1} with $M_0$ and $\rho_{M_0}$ independent of $\eps$.
Let $(\rho^\eps,u^\eps)$ for $\eps\in(0,\eps_0)$ be an entropy solution
of \eqref{eq:relativistic-Euler}, determined by Theorem {\rm \ref{thm:relativistic-main}} above,
with light speed $c=\frac{1}{\sqrt{\eps}}$
and initial data $(\rho^\eps_0,u_0^\eps)\in (L^\infty(\mathbb{R}))^2$ with $\rho_0^\eps\geq 0$ such that
\begin{equation}\label{ID-2}
0\le c_{\rm s}(\rho^\eps_0(x)),\, |u_0^\eps(x)|
<\frac{1}{\sqrt{\eps}} \qquad\; \text{ for all $\eps\in(0,\eps_0)$ and a.e. $x\in\R$},
\end{equation}
and $(\rho_0^\eps, u_0^\eps)\to (\rho_0,u_0)$ a.e. as $\eps\to0$.
Then there exist $M>0$ and $\rho_M$, independent of $\eps$, such that
$$
|u^\eps(t,x)|\leq M< \frac{1}{\sqrt{\eps}}, \hspace{3mm}    0\leq\rho^\eps(t,x)\leq \rho_M <\rho_{\max}^\eps
\qquad\, \text{ for a.e. $(t,x)\in\mathbb{R}^2_+$},
$$
and,  up to a subsequence, $(\rho^\eps,u^\eps)\rightarrow(\rho,u)$ a.e. and in $L^r_{\loc}(\mathbb{R}^2_+)$ for all $r\in[1,\infty)$ as $\eps\to 0$,
where $(\rho,u)$ is an entropy solution of the classical Euler equations \eqref{eq:classical-Euler}
with initial data $(\rho_0,u_0)$
satisfying
$$
|u(t,x)|\leq M, \hspace{3mm}  0\leq\rho(t,x)\leq \rho_M  \qquad\, \text{ for a.e. $(t,x)\in\mathbb{R}^2_+$}.
$$
\end{thm}

We remark that an alternative $2\times2$ system of conservation laws in the theory of special relativity
(also sometimes called the relativistic Euler equations in the literature)
is the following system of conservation laws of energy and momentum:
\begin{equation}\label{eq:relativistic-Euler-alternative}
\begin{cases}
\partial_t\big(\rho+\frac{\eps(\rho+\eps p)u^2}{1-\eps u^2}\big) + \partial_x\big(\frac{(\rho+\eps p)u}{1-\eps u^2}\big) = 0,\\[2mm]
\partial_t\big(\frac{(\rho+\eps p)u}{1-\eps u^2}\big) +  \partial_x\big(\frac{(\rho+\eps p)u^2}{1-\eps u^2}+ p(\rho)\big) =0.
\end{cases}
\end{equation}
System \eqref{eq:relativistic-Euler-alternative} has the same eigenvalues and Riemann invariants as
those for \eqref{eq:relativistic-Euler}, which implies that
the governing entropy equation for system \eqref{eq:relativistic-Euler-alternative} is the same as
that for \eqref{eq:relativistic-Euler}, so that our analysis of the entropy functions for \eqref{eq:relativistic-Euler}
and the associated compactness framework are also extended to the alternative system, \eqref{eq:relativistic-Euler-alternative}.
Therefore, we also obtain the following theorem.

\begin{thm}\label{thm:alternativesystem}
Let $(\rho_0,u_0)\in (L^\infty(\mathbb{R}))^2$ with $\rho_0\geq 0$, and let
the pressure function $p(\rho)$ satisfy  \eqref{ass:stricthyperbolicity}--\eqref{ass:genuine-nonlinearity}
for $\rho>0$ and \eqref{ass:pressure} for some $\gamma\in (1,3)$.  Then the following statements hold{\rm :}
\begin{enumerate}
\item[\rm (i)]
Let $(\rho_0,u_0)$ satisfy \eqref{ID-1}.
Then there is $\eps_0>0$ such that, if $\eps\leq\eps_0$,
there exists an entropy solution $(\rho^\eps,u^\eps)$ of \eqref{eq:relativistic-Euler-alternative} satisfying
$$
|u^\eps(t,x)|\leq M<\frac{1}{\sqrt{\eps}},  \hspace{3mm}   0\leq\rho^\eps(t,x)\leq \rho_M<\rho_{\max}^\eps
\qquad \text{ for a.e. $(t,x)\in\R^2_+$,}
$$
for some constants $M$ and $\rho_M$ depending only on the initial data, but independent of $\eps$.
Furthermore, for any fixed $\eps>0$, the solution operator $(\rho^\eps, u^\eps)(t,\cdot), t>0$,
determined by the above is compact in $L^r_{\rm loc}(\mathbb{R}^2_+)$ for $1\leq r <\infty$.

\smallskip
\item[\rm (ii)]
Let  $(\rho^\eps,u^\eps)$ for $\eps\in(0,\eps_0)$ is an entropy solution
of \eqref{eq:relativistic-Euler}, determined by {\rm (i)} above, with light speed $c=\frac{1}{\sqrt{\eps}}$
and initial data $(\rho^\eps_0,u_0^\eps)\in (L^\infty(\mathbb{R}))^2$ with $\rho_0^\eps\geq 0$
satisfying \eqref{ID-2} such that $(\rho_0^\eps, u_0^\eps)\to (\rho_0,u_0)$ a.e. as $\eps\to0$.
Then,
up to a subsequence, $(\rho^\eps,u^\eps)\rightarrow(\rho,u)$ a.e. and in $L^r_{\loc}(\mathbb{R}^2_+)$ for all $r\in[1,\infty)$ as $\eps\to 0$,
where $(\rho,u)$ is an entropy solution of the classical Euler equations \eqref{eq:classical-Euler}
with initial data $(\rho_0,u_0)$
satisfying
$$
|u(t,x)|\leq M, \hspace{3mm}  0\leq\rho(t,x)\leq \rho_M \qquad \text{ for a.e. $(t,x)\in\mathbb{R}^2_+$}
$$
for some constants $M$ and $\rho_M$.
\end{enumerate}
\end{thm}

Before we describe our approach for the proofs of these results,
we recall the situation for the classical Euler equations \eqref{eq:classical-Euler}.
DiPerna \cite{DiPerna2} first showed the existence of entropy solutions of \eqref{eq:classical-Euler}
for the case of a gamma-law gas with $\ga=1+\frac{2}{N}$, $N$ odd and $N\geq 5$,
by developing the method of compensated compactness of Murat--Tartar \cite{Murat1,Tartar}.
The general case  $\ga\in(1,\frac{5}{3}]$ for polytropic gases was first solved in Chen \cite{Chen1} and Ding--Chen--Luo \cite{DingChenLuo}
by developing new techniques for entropy analysis which involve fractional derivatives and the Hilbert transform,
combined with the compensated compactness argument.
The case $\ga\geq 3$ was subsequently solved by Lions--Perthame--Tadmor \cite{LionsPerthameTadmor} through the introduction of the kinetic
formulation, before Lions--Perthame--Souganidis \cite{LionsPerthameSouganidis} solved the problem for the remaining interval $\ga\in(\frac{5}{3},3)$,
simplifying the proof for all $\ga\in(1,3)$.
Chen--LeFloch \cite{ChenLeFloch, ChenLeFloch2} solved the case of a more general pressure law,
under the assumptions of strict hyperbolicity and genuine nonlinearity away from the vacuum and an approximate gamma-law form close to the vacuum;
see \cite{ChenLeFloch,ChenLeFloch2}, as well as \eqref{ass:pressure}, for the precise assumptions on the pressure law.

The procedure that we undertake to establish the existence of solutions to the relativistic Euler equations \eqref{eq:relativistic-Euler}
is motivated by the works
for the classical Euler equations described above.
We construct a sequence of approximate solutions to the equations via a vanishing viscosity method
and pass the viscosity to zero.
As system \eqref{eq:relativistic-Euler} admits an invariant region,
we obtain the uniform bounds in $L^\infty$ of the approximate solutions.
Passing to a weak-star limit in $L^\infty$, we then associate a Young measure $\{\nu_{t,x}\}$
to the sequence,
characterizing the weak convergence.
As is well known, such a weak convergence is insufficient to pass to a limit in
the nonlinear terms of the equations, and hence we apply the compensated compactness argument
to improve this convergence.

Applying the method of compensated compactness
with the uniform estimates of the approximate solutions,
we deduce the Tartar commutation relation:
$$
\langle \nu_{t,x}, \eta_1 q_2-\eta_2 q_1\rangle
=\langle \nu_{t,x}, \eta_1\rangle \langle \nu_{t,x},q_2\rangle-\langle \nu_{t,x},\eta_2\rangle \langle \nu_{t,x},q_1\rangle
$$
for all weak entropy pairs $(\eta_1,q_1)$ and $(\eta_2,q_2)$ defined in \S 3 below.
We then show that this relation is sufficient to argue that the support of the probability measure $\nu_{t,x}$ reduces
to a single point and hence deduce the strong convergence of the approximate solutions
{\it a.e.} and in $L^p_{\loc}$.

To complete this reduction argument, we require a thorough understanding of the entropy pairs
for system \eqref{eq:relativistic-Euler}.
To this end, we establish the existence of fundamental kernels generating the admissible entropy pairs.
In order to do this, we make an \it ansatz \rm for the leading order behavior of the entropy kernels
close to the vacuum and take asymptotic expansions around the leading order terms.
This leaves us with an equation for the remainder that is then solved via a fixed point argument.
We establish estimates on both the leading terms and the remainder to demonstrate their respective
regularity properties.
With the obtained expansions, we analyze the singularities of the kernels and exploit properties
of cancellation of singularities in the commutation relation to conclude our arguments.
As a by-product, we also obtain the compactness of the solution operator in $L^p_{\loc}$.

Finally, we exploit the relationship of the relativistic entropy kernels to the classical entropy kernels
to demonstrate the convergence of the Newtonian limit.
Applying the compactness framework developed for the classical Euler equations in \cite{ChenLeFloch,ChenLeFloch2},
we gain the strong convergence of the relativistic solutions to the classical solutions of the Euler equations.

The structure of the paper is as follows:
In \S \ref{sec:relativisticproperties},
we analyze some basic properties of the relativistic Euler equations.
We then introduce the definitions of the entropy and entropy-flux kernels,
and state our main theorems concerning the existence and regularity of these kernels in \S \ref{sec:relativisticentropy}.
The proofs of these theorems are provided in \S \ref{sec:entropykernel}--\S\ref{sec:entropyfluxkernel}.
Moreover, these sections provide a detailed analysis of the asymptotics of the kernels as the density
approaches the vacuum state ({\it i.e.}, $\rho\to0$), as well as the singularities of their derivatives.
After this, in \S \ref{sec:relativisticcompactness},
we establish a compactness framework for approximate or exact solutions of
both systems \eqref{eq:relativistic-Euler} and \eqref{eq:relativistic-Euler-alternative}
via
a careful analysis of the Tartar commutation relation
for the relativistic entropies constructed from the kernels, established in \S\ref{sec:relativisticentropy}.
This analysis exploits the properties of cancellation of the singularities in the entropy and entropy-flux kernels,
relying on the expansions established in the earlier sections.
In \S\ref{sec:relativisticartificial}, we outline the construction of the artificial viscosity solutions
and demonstrate that they satisfy the compactness framework.
This allows us to conclude the first main theorem, Theorem \ref{thm:relativistic-main},
as well as Theorem \ref{thm:alternativesystem}, in \S\ref{sec:mainresultproof}.
Finally, in \S\ref{sec:Newtonianlimit}, we prove our second main theorem, Theorem \ref{thm:Newtonianlimit},
concerning the Newtonian limit of a sequence of solutions of the relativistic Euler equations
to the classical Euler equations.

\section{Basic Properties}\label{sec:relativisticproperties}
In this section, we analyze some basic properties of system \eqref{eq:relativistic-Euler}.
Writing $U(\rho,u)$ as in \eqref{eq:conservedvariables} for the conserved variables, we calculate
\begin{equation} \label{eq:gradF}
\nabla_U F(U) = \nabla_{(\rho,u)}F(U)(\nabla_{(\rho,u)}U)^{-1}
= \begin{pmatrix}
\frac{\eps (u^2 - p')u}{1 - \eps^2 p'u^2} & \frac{n(1-\eps u^2)^{\frac{3}{2}}}{(\rho + \eps p)(1-\eps^2 p'u^2)} \\[2mm]
\frac{(p' - u^2)(\rho + \eps p)\sqrt{1 - \eps u^2}}{n(1-\eps^2 p'u^2)} & \frac{(2 - \eps p' - \eps u^2)u}{1 -\eps^2 p'u^2}
\end{pmatrix},
\end{equation}
where $\nabla_U$ and $\nabla_{(\rho,u)}$ denote the gradients
in variables $U$ and $(\rho,u)$, respectively.
Then the eigenvalues of $\nabla_U F(U)$ are
\begin{equation*}
\lambda_- = \frac{u - \sqrt{p'(\rho)}}{1 - \eps u\sqrt{p'(\rho)}}, \quad
\lambda_+ = \frac{u + \sqrt{p'(\rho)}}{1 + \eps u\sqrt{p'(\rho)}},
\end{equation*}
and the corresponding eigenvectors are
\begin{align*}
 r_-=\begin{pmatrix}
       1\\[1mm]
       \frac{(\rho+\eps p)(u-\sqrt{p'})}{n\sqrt{1-\eps u^2}}
      \end{pmatrix},\quad
 r_+=\begin{pmatrix}
      1\\[1mm]
      \frac{(\rho+\eps p)(u+\sqrt{p'})}{n\sqrt{1-\eps u^2}}
     \end{pmatrix}.
\end{align*}
The sound speed in the fluid is given by $c_{\rm s}(\rho)=\sqrt{p'(\rho)}$, and we henceforth assume that
$$
\sqrt{p'(\rho)}\leq c=\frac{1}{\sqrt{\eps}}.
$$
As defined earlier, $\rho_{\max}^\eps$ is such that $c_{\rm s}(\rho_{\max}^\eps)=\frac{1}{\sqrt{\eps}}$
if a finite $\rho_{\max}^\eps$ exists,
or $\infty$ otherwise.
Then, in the region:
$$
\{|u|< \frac{1}{\sqrt{\eps}},\, 0<\rho<\rho_{\max}^\eps\},
$$
we see that $\lambda_+ - \lambda_- > 0$ so that the system is strictly hyperbolic.

The Riemann invariants of the system are
$$
w := v + k,\quad z:= v - k,
$$
where
\begin{equation}\label{eq:v(u)}
v = v(u) := \frac{1}{2\sqrt{\eps}}\log(\frac{1 + \sqrt{\eps}u}{1-\sqrt{\eps}u})
\end{equation}
and
\begin{equation}
k = k(\rho) := \int_{0}^{\rho}\frac{\sqrt{p'(s)}}{s + \eps p(s)}\dd s.
\end{equation}
Note that the mapping: $u \mapsto v(u)$ is a smooth, increasing bijection
from $(-\frac{1}{\sqrt{\eps}},\frac{1}{\sqrt{\eps}})$ to $\mathbb{R}$, and that
$\rho \mapsto k(\rho)$ is a smooth, increasing bijection from $(0,\rho_{\max}^\eps)$ onto its image.
For the inverse of $v$, we write $u$ as
\beq\label{eq:u(v)}
u(v):=\frac{1}{\sqrt{\eps}}\tanh(\sqrt{\eps} v).
\eeq
As mentioned earlier, to close the system, we impose an equation of state, {\it i.e.}, a general pressure law,
which satisfies conditions \eqref{ass:stricthyperbolicity}--\eqref{ass:genuine-nonlinearity} for $\rho>0$
and \eqref{ass:pressure} close to the vacuum.

We compare the nonlinear function $k(\rho)$ to the equivalent function for the classical Euler equations
equipped with a gamma-law pressure ({\it cf.}~\cite{ChenLeFloch}), for which $k(\rho)=\rho^\th$ with $\theta=\frac{\gamma-1}{2}$.
With assumption \eqref{ass:pressure} on the pressure,
we observe the following behavior of $k(\rho)$ near the vacuum.
For ease of reference, we state this as a lemma.

\begin{lemma}\label{lemma:k-asymptotics}
As $\rho\to0$, the nonlinear function $k(\rho)$ and its first derivative obey the following asymptotics{\rm :}
\beqa\label{eq:k-asymptotics}
&k(\rho) = \rho^{\theta} + O(\rho^{3\theta}) \qquad \text{ as $\rho\rightarrow0$,}\\
&k'(\rho) = \frac{\sqrt{p'(\rho)}}{\rho + \eps p(\rho)} =\theta\rho^{\theta-1}+O(\rho^{3\theta-1}) \qquad \text{ as $\rho\rightarrow0$.}
\eeqa
Moreover, its derivatives $k^{(n)}(\rho)$, for $n= 2,3$, can be expanded as
\beqas
k''(\rho) = \theta(\theta-1)\rho^{\theta-2}+O(\rho^{3\theta-2}),
\,\,\,\, k^{(3)}(\rho) = \theta(\theta-1)(\th-2)\rho^{\theta-3}+O(\rho^{3\theta-3}) \qquad \mbox{as $\rho\to 0$}.
\eeqas
\end{lemma}

We define another exponent: $\lambda = \frac{3-\gamma}{2(\gamma -1)}>0$  for the use in the next section.
Note that $\la$ is related to $\th$ by the relation: $2\lambda\theta = 1-\theta$.

An analysis of system \eqref{eq:relativistic-Euler-alternative} shows that it has also the same eigenvalues $\la_-$ and $\la_+$,
and the same Riemann invariants
$w=v(u)+k(\rho)$ and $z=v(u)-k(\rho)$,  as defined above.

\section{Entropy Pairs and Entropy Solutions}\label{sec:relativisticentropy}
In order to analyze the limit of our approximate solutions of system \eqref{eq:relativistic-Euler}
and prove the strong convergence of the sequence,
we first need to understand the structure and behavior of entropy pairs of the system.
Therefore, the purpose of this section is to provide the basis and framework for this analysis.

\begin{defn}\label{def:relativistic-entropy-pair}
An \textit{entropy pair} $(\eta,q)$ for system \eqref{eq:relativistic-Euler}
is a pair of $C^1$ entropy and entropy-flux functions satisfying the relation:
$$
\nabla\eta(U)\nabla F(U) =\nabla q(U).
$$
A \textit{weak entropy} $\eta$ is an entropy that vanishes at the vacuum state: $\eta|_{\rho = 0}=0$.
\end{defn}

We observe that an equivalent characterization of the entropy pair in the Riemann invariant coordinates $(w,z)$ is given by
\beq
q_w=\la_+\,\eta_w,\qquad q_z=\la_-\,\eta_z.
\eeq
In particular, as the eigenvalues and Riemann invariants of the two systems \eqref{eq:relativistic-Euler}
and \eqref{eq:relativistic-Euler-alternative} coincide, we deduce that the two systems share the same entropy and entropy-flux functions.
We therefore restrict our attention  to system \eqref{eq:relativistic-Euler} in the sequel.

\begin{defn}\label{def:relativistic-entropy-solution}
 A pair of bounded, measurable functions $(\rho,u)$ such that
 $$
 |u| < c=\frac{1}{\sqrt{\eps}}, \quad 0\leq\rho<\rho_{\max}^\eps
 $$
 is an \it entropy solution \rm of the Cauchy problem \eqref{eq:relativistic-Euler}--\eqref{Cauchy-problem}
 provided that
 \begin{enumerate}
 \item[\rm (i)] For any $\phi \in C_{\rm c}^1(\mathbb{R}_+^2)$,
 \begin{align*}
 &\iint_{\R_+^2}
  \Big( \frac{n}{\sqrt{1-\eps u^2}}\,\phi_t + \frac{nu}{\sqrt{1-\eps u^2}}\,\phi_x\Big)\dd x\,\dd t
  + \int_\R \frac{n_0}{\sqrt{1-\eps u_0^2}}\,\phi(0,x)\,\dd x =0,\\[2mm]
 &\iint_{\R_+^2} \Big(\frac{(\rho + \eps p)u}{1-\eps u^2}\,\phi_t
   + \big(\frac{(\rho + \eps p)u^2}{1-\eps u^2}+p(\rho)\big)\phi_x\Big)\dd x\,\dd t
    + \int_\R \frac{(\rho_0 + \eps p(\rho_0))u_0}{1-\eps u_0^2}\,\phi(0,x)\,\dd x =0;
 \end{align*}

\medskip
 \item[\rm (ii)] For any nonnegative function $\phi \in C_{\rm c}^1(\mathbb{R}_+^2)$
  and $C^1$ weak entropy pair $(\eta, q)(\rho,u)$ with $\eta$ convex with respect to $U$,
$$
\iint_{\R_+^2}\big(\eta(\rho,u)\partial_t\phi + q(\rho,u)\partial_x\phi\big)\,\dd x\,\dd t
+ \int_\R\eta(\rho_0, u_0)\phi(0,x)\,\dd x \geq 0,
$$
where $(\eta, q)(\rho,u):=(\eta, q)(U(\rho, u))$.
 \end{enumerate}
\end{defn}

An explicit entropy pair is given by
\beq\label{def:physicalentropy}
\eta^*(U(\rho,u))=\frac{\rho+\eps^2 pu^2}{1-\eps u^2},\qquad q^*(U(\rho,u))=\frac{(\rho+\eps p)u}{1-\eps u^2}.
\eeq
Then
\begin{equation}
 \nabla^2\eta^*(U)
 =\alpha_0(\rho,u)
 \begin{pmatrix}
 \frac{\eps(\rho+\eps p)(p'+u^2+2\eps p' u^2)}{n(1-\eps u^2)} & -\frac{\eps(1+\eps p')u}{\sqrt{1-\eps u^2}}\\[2mm]
 -\frac{\eps(1+\eps p')u}{\sqrt{1-\eps u^2}} & \frac{\eps n}{\rho+\eps p}
\end{pmatrix}
 \end{equation}
with $\alpha_0(\rho,u)=\frac{(1-\eps u^2)^2}{n(1-\eps^2 p'u^2)}>0$. In particular, $\eta^*(U)$ is a convex entropy.

We remark that the entropy pair \eqref{def:physicalentropy} is actually the first conservation law in the
alternative system \eqref{eq:relativistic-Euler-alternative}.

We begin our analysis of the entropy functions of \eqref{eq:relativistic-Euler} by constructing a fundamental solution of the \it entropy equation\rm.

\subsection{Entropy Equation}
Let $(\eta,q)$ be an entropy pair. Then it follows from \eqref{eq:gradF} and Definition \ref{def:relativistic-entropy-pair} that
\begin{align}
  q_\rho &=\eta_\rho\frac{u(1-\eps p')}{1-\eps^2 p'u^2} + \eta_u\frac{(1-\eps u^2)^2p'}{(\rho+\eps p)(1-\eps^2 p'u^2)},\\[1mm]
  q_u&=\eta_\rho\frac{\rho + \eps p}{1-\eps^2 p'u^2} + \eta_u\frac{u(1-\eps p')}{1-\eps^2 p'u^2}.
\end{align}
Eliminating $q$ and changing the coordinate: $u \mapsto v$, as in \eqref{eq:v(u)}, yield
\begin{equation}\label{eq:A,Bintro}
\eta_{\rho \rho} - k'(\rho)^2\eta_{vv} + \eps A(\rho,v)\eta_{\rho} + \eps B(\rho,v)v\eta_v = 0,
\end{equation}
where
\begin{align*}
&A(\rho,v)=\frac{2p'(\rho)}{\rho + \eps p(\rho)}\frac{1-\eps u^2\big(1-\frac{p''(\rho)(\rho+\eps p)}{2p'(\rho)}\big)}{1-\eps^2 p'(\rho)u^2}, \\[1mm]
&B(\rho,v)=\frac{2up'(\rho)\big(1-\eps p'(\rho)-\frac{p''(\rho)(\rho+\eps p(\rho))}{2p'(\rho)}\big)}{v(u)(\rho+\eps p(\rho))^2(1-\eps^2 p'(\rho)u^2)}.
\end{align*}
To simplify notation, we use the operator:
\beq
\mathbf{L}:=\partial_{\rho\rho}-k'(\rho)^2\partial_{vv}+\eps A(\rho,v)\partial_\rho+\eps B(\rho,v)v\partial_v.
\eeq

\begin{defn}
The \textit{entropy kernel} $\chi=\chi(\rho,v,s)$ is the unique solution of the equation:
\begin{equation}
\begin{cases} \label{eq:entropyequation}
\mathbf{L}\chi=\chi_{\rho \rho} - k'(\rho)^2\chi_{vv} + \eps A(\rho,v)\chi_{\rho} + \eps B(\rho,v)v\chi_v = 0, \\[1mm]
\chi\rvert_{\rho=0} = 0, \\[1mm]
\chi_{\rho}\rvert_{\rho = 0} = \delta_{v=s}\qquad \mbox{for $s\in\R$}.
\end{cases}
\end{equation}

\end{defn}
We recall that \eqref{eq:relativistic-Euler} is invariant under the Lorentz
transformations:
$$
(t,x)\to
(t', x')=(\frac{t-\eps \tau x}{\sqrt{1-\eps \tau^2}},\frac{x-\tau t}{\sqrt{1-\eps \tau^2}}) \qquad \mbox{for $|\tau|<\frac{1}{\sqrt{\eps}}$}.
$$
Under this transformation, velocity $u$ and the associated function $v$ also transform as:
$$
u'=\frac{u-\tau}{1-\eps \tau u}, \quad v':=v(u')=v(u)-v(\tau).
$$
By the invariance of the equations under these transformations, the entropy equation
is also  invariant under such Lorentz shifts.
Thus, for $s=v(\tau)$,
$$
\chi(\rho,v,s)=\chi(\rho,v-s,0)=\chi(\rho,0,s-v),
$$
so that it suffices to solve $\chi$ in the case that $s=0$.
We therefore write $\chi(\rho,v-s)=\chi(\rho,v,s)$ henceforth in a slight abuse of notation.

The kernel provides a representation formula for weak entropies of system \eqref{eq:relativistic-Euler}.
That is, any weak entropy function can be represented by convolution with a test function $\psi(s)$ as
\beqs
\eta^\psi(\rho,u)=\int_\R\chi(\rho,v(u)-s) \psi(s)\,\dd s.
\eeqs
Before we continue, it is worth making an aside at this point to compare the situation to the classical Euler
equations \eqref{eq:classical-Euler}.
For system \eqref{eq:classical-Euler}, the entropy equation is the simpler equation:
\beq\label{eq:classicalentropyeq}
\chi^*_{\rho\rho}-k'(\rho)^2\chi^*_{vv}=0.
\eeq
For the gamma-law gas, $k'(\rho)=\th\rho^{\th-1}$, and \eqref{eq:classicalentropyeq} has the fundamental solution:
\beq\label{eq:chistar}
\chi^*(\rho,v)=M_\la [\rho^{2\th}-v^2]_+^\la,
\eeq
where $\la>0$ is defined as in \S \ref{sec:relativisticproperties} and $M_\la>0$ is a constant depending only on $\la$.

With this as a motivation, we make an ansatz for the entropy kernel of system \eqref{eq:relativistic-Euler} in the form:
\begin{equation}\label{eq:relativisticentropyexpansion}
\chi(\rho,v) = a_1(\rho)[k(\rho)^2 - v^2]_+^{\lambda} + a_2(\rho)[k(\rho)^2 - v^2]_+^{\lambda+1} + g(\rho,v).
\end{equation}
By the principle of finite propagation speed,
we expect the remainder function $g(\rho,v)$ to have the same support as the first two terms.

In anticipation of the next theorem, we recall the definition of fractional derivatives:
For a function $f=f(s)$ of compact support, the fractional derivative of order $\mu>0$ is
\beq\label{def:fractionalderivative}
\partial_s^\mu f=\Gamma(-\mu)f* [s]_+^{-\mu-1},
\eeq
where $\Gamma(\cdot)$ is the Gamma function.

Henceforth, we suppose that the density has a fixed upper bound: $\rho\leq\rho_M$ so that $\rho_M<\rho_{\max}^\eps$.
The universal constant $C>0$ is independent of $\rho$, but may depend on $\rho_{M}$.

\begin{thm}[Relativistic Entropy Kernel]\label{thm:entropykernelmain}
The entropy kernel admits the expansion
\begin{equation}\label{eq:kernelexpansion}
\chi(\rho,v) = a_1(\rho)[k(\rho)^2 - v^2]_+^{\lambda} + a_2(\rho)[k(\rho)^2 - v^2]_+^{\lambda+1} + g(\rho,v),
\end{equation}
where the coefficients $a_1(\rho)$ and $a_2(\rho)$ are such that, when $0\leq\rho\leq\rho_M$,
\begin{equation}
a_1(\rho) = c_{*,\lambda}k(\rho)^{-\lambda}k'(\rho)^{-\frac{1}{2}}e^{\tilde{a}(\rho)}>0,
\end{equation}
 and
\beq
a_1(\rho)+|a_2(\rho)|\leq C,
\eeq
with
\begin{equation} \label{3.14a}
\tilde{a}(\rho)=\frac{\eps}{2}\int_0^\rho\big(-A(s, k(s))+\frac{k(s)}{k'(s)}B(s, k(s))\big)\dd s
\end{equation}
as defined in \eqref{4.1a} below,
and
$c_{*,\la}>0$ being a constant depending only on $\lambda$.
Moreover, the remainder function $g(\rho,v)$ and its derivatives $\partial^\mu_v g(\rho,v)$ are H\"{o}lder continuous for $0<\mu<\la+2$, and
satisfy that, for $0<\be<\mu$,
$$
|\partial_v^\be g(\rho,v)|\leq C\rho^{1+ (1-2\mu+\be)\th}[k(\rho)^2-v^2]_+^{\mu-\be}.
$$
\end{thm}

By definition, to each entropy function is associated a corresponding entropy-flux function.
These entropy-flux functions are generated by another kernel, the {\it entropy-flux kernel} $\sigma(\rho,v,s)$.
\begin{defn}
The entropy-flux kernel is defined by
\begin{equation}
\begin{cases} \label{eq:entropyfluxequation}
\mathbf{L}\sigma:=\sigma_{\rho \rho} - k'(\rho)^2\sigma_{vv} + \eps A(\rho,v)\sigma_{\rho} + \eps B(\rho,v)v\sigma_v = F(\rho,v),\\[1mm]
\sigma\rvert_{\rho=0} = 0, \\[1mm]
\sigma_{\rho}\rvert_{\rho = 0} = \frac{u(1-\eps p')}{1-\eps^2 p'u^2}\delta_{v=s},
\end{cases}
\end{equation}
where $F(\rho,v)$ is given explicitly later in \eqref{eq:fluxequationrighthandside}.
\end{defn}
The entropy function generated by the convolution of a test function $\psi(s)$ with the entropy kernel
has a corresponding entropy-flux given by
$$
q(\rho,v)=\int_\R\sigma(\rho,v,s)\psi(s)\,\dd s.
$$
As we have seen for the entropy equation, the equation in \eqref{eq:entropyfluxequation}
is invariant under the Lorentz transformation, but the initial conditions for the entropy flux kernel are not.
We therefore consider, instead of $\sigma$, the difference $\sigma-\frac{u(1-\eps p')}{1-\eps^2 p'u^2}\chi$.
Writing $\tilde{u}(\rho,v)=\frac{u(1-\eps p')}{1-\eps^2 p'u^2}$,
this difference satisfies the following initial value problem:
\begin{equation}
\begin{cases} \label{eq:modifiedentropyfluxequation}
\mathbf{L}(\sigma-\tilde{u}\chi)= \tilde{F}(\rho,v),\\
(\sigma-\tilde{u}\chi)\rvert_{\rho=0} = 0,\\
(\sigma-\tilde{u}\chi)_{\rho}\rvert_{\rho = 0} =0,
\end{cases}
\end{equation}
with $\tilde{F}(\rho,v)$ defined by
$$
\tilde{F}(\rho,v):=F(\rho,v)-\mathbf{L}(\tilde{u}\chi).
$$
Then problem \eqref{eq:modifiedentropyfluxequation} is Lorentzian invariant so that
$$
(\sigma-\tilde{u}\chi)(\rho,v,s)=(\sigma-\tilde{u}\chi)(\rho,v-s,0)=(\sigma-\tilde{u}\chi)(\rho,0,s-v).
$$
Therefore, it suffices to solve $\sigma-\tilde{u}\chi$ for the case: $s=0$.

\begin{thm}[Relativistic Entropy-Flux Kernel]\label{thm:entropyfluxkernelmain}
The entropy-flux kernel admits the expansion{\rm :}
\begin{equation}\label{eq:relativisticentropyfluxexpansion}
(\sigma-\tilde{u}\chi)(\rho,v) = -v\big(b_1(\rho)[k(\rho)^2 - v^2]_+^{\lambda} + b_2(\rho)[k(\rho)^2 - v^2]_+^{\lambda+1}\big) + h(\rho,v),
\end{equation}
where coefficients $b_1(\rho)$ and $b_2(\rho)$ satisfy that, for $0\leq\rho\leq\rho_M$,
\beqa
b_1(\rho)>0,\quad b_1(\rho)+|b_2(\rho)|\leq C.
\eeqa
Moreover, the remainder function $h(\rho,v)$ and its derivatives $\partial^\mu_v h(\rho,v)$ are H\"{o}lder continuous for $0<\mu<\la+2$ and satisfy that,
for $0<\be<\mu$,
$$
|\partial_v^\be h(\rho,v)|\leq C\rho^{1+(1-2\mu+\be)\th}[k(\rho)^2-v^2]_+^{\mu-\be}.
$$
\end{thm}

\section{The Weak Entropy Kernel}\label{sec:entropykernel}

This section is devoted to the proof of Theorem \ref{thm:entropykernelmain}
to show the existence of the entropy kernel and examine its regularity properties.

\subsection{Roadmap for the construction of the entropy kernel}
In order to aid the comprehension of the reader, we first provide a summary of the structure of the proof and construction
of the entropy kernel contained in \S\ref{sec:entropykernelcoeffs}--\ref{sec:entropykernelproofs}.
A similar structure holds for the construction of the entropy flux kernel in \S\ref{sec:entropyfluxkernel},
as given in Theorem \ref{thm:entropyfluxkernelmain}.

We begin by making an ansatz of form \eqref{eq:kernelexpansion} for the structure of the entropy kernel $\chi(\rho,v)$.
Substituting this ansatz into the equation in \eqref{eq:entropyequation},
we identify the most singular terms resulting from this expansion.
Enforcing the cancellation of these terms,
we derive an ordinary differential equation (ODE) for the first coefficient $a_1(\rho)$,
the solution for which is given in \eqref{def:a1}.
This solution is normalized in order to satisfy the initial condition $\chi_\rho\big|_{\rho=0}=\de_{v=0}$.
Arguing similarly for the next most singular terms (using the chosen function $a_1(\rho)$),
we obtain a further ODE for $a_2(\rho)$, the solution for which is given in \eqref{eq:a_2coefficient}.

The next step is to derive a suitable equation for the remainder term $g(\rho,v)$.
To do this, we identify an operator $\widetilde{\mathbf{L}}$ capturing the principal part of
equation \eqref{eq:entropyequation} but with coefficients independent of $v$.
Employing the Fourier transform with respect to $v$,
we obtain a formal equation for the remainder term $\widehat{g}(\rho,\xi)$ in Proposition \ref{prop:Pgexpression}.
A key observation is that the choice of $a_1(\rho)$ and $a_2(\rho)$ to cancel the leading order singularities
is manifested on the Fourier side by the cancellation of the most slowly decaying terms in the equation
for $\widehat{g}(\rho,\xi)$ as $\xi\to\pm\infty$.
This leads to the higher order H\"older regularity of $g(\rho,v)$ so that $g(\rho,v)$ can be treated as a true remainder term.
Using the fundamental solutions of operator $\widetilde{\mathbf{L}}$ in Lemma \ref{lemma:Ltildefundamental},
we may employ the variation of parameters to find a formal representation formula for $\widehat{g}(\rho,\xi)$; see \eqref{eq:ghatrep}.

In Theorem \ref{thm:entkernelexistence}, we give a proof of the existence of a solution $\widehat{g}(\rho,\xi)$
satisfying the representation formula via a constructive fixed point argument, before proving the higher regularity
of the obtained fixed point
in Theorem \ref{thm:entropykernelregularity}.
It is then straightforward to see that the obtained function $g(\rho,v)$ from the inverse Fourier transform satisfies the desired equation,
which leads to the existence of the entropy kernel claimed in Theorem \ref{thm:entropykernelmain}.

The final stage is to quantify the difference between the entropy kernel $\chi(\rho,v)$ and the Newtonian entropy
kernel $\chi^*(\rho,v)$ defined in \eqref{eq:chistar} and the following equation.
This is done in Theorem \ref{thm:chipositive}, again by using the variation of parameters formula in the Fourier space.

\subsection{The coefficients for the entropy kernel}\label{sec:entropykernelcoeffs}
As a preliminary observation, we note that, by the principle of finite speed of propagation,
coefficients $A$ and $B$ may be redefined to be $0$ outside the support of $\chi(\rho,v-s)$,
which is the set $\{|v-s|\leq k(\rho)\}$.
Moreover, as $A$ and $B$ are functions of $v^2$, a simple Taylor expansion around $v^2=k(\rho)^2$ gives

\begin{lemma} \label{lemma:A,Bdecomposition}
The coefficient functions $A$ and $B$ in equation \eqref{eq:entropyequation} for the entropy kernel can be written as:
\begin{align*}
&A(\rho,v) = A_0(\rho)\mathbb{1}_{|v|\leq k(\rho)} + \eps A_1(\rho)[k(\rho)^2 - v^2]_{+} + \eps^2 A_2(\rho,v)[k(\rho)^2 - v^2]_{+}^2, \\[1mm]
&B(\rho,v) = B_0(\rho)\mathbb{1}_{|v|\leq k(\rho)} + \eps B_1(\rho)[k(\rho)^2 - v^2]_{+} + \eps^2 B_2(\rho,v)[k(\rho)^2 - v^2]_{+}^2,
\end{align*}
with
\begin{align*}
|A_0(\rho)| + |A_1(\rho)| + |A_2(\rho,v)|
+\rho \big(|B_0(\rho)| + |B_1(\rho)| + |B_2(\rho,v)|\big) \leq C\rho^{\gamma - 2},
\end{align*}
and $\mathbb{1}_{|v|\leq k(\rho)}=1$ when $|v|\leq k(\rho)$ and $0$ when $|v|>k(\rho)$.
\end{lemma}

With the ansatz for the entropy kernel,
we determine  coefficients $a_1(\rho)$ and $a_2(\rho)$.
To do this, we substitute ansatz \eqref{eq:relativisticentropyexpansion} into the entropy equation \eqref{eq:entropyequation}
and examine the most singular terms.
By choosing the coefficients such that these singular terms vanish,
we are able to solve the equation for the higher order remainder.

Denote
$$
G_\nu(\rho,v):=[k(\rho)^2-v^2]_+^\nu \qquad \mbox{for $\nu\in \R$}.
$$
Then we have the following identities:
$$
\partial_\rho G_\nu(\rho,v)=2\nu k(\rho)k'(\rho)G_{\nu-1}(\rho,v),
\quad \partial_v G_\nu(\rho,v)=-2\nu vG_{\nu-1}(\rho,v).
$$
Substituting \eqref{eq:relativisticentropyexpansion} into \eqref{eq:entropyequation} and grouping the terms yield:
\beqas
\chi_{\rho\rho}-k'(\rho)^2\chi&_{vv}+\eps A\chi_\rho+\eps Bv\chi_v\\
=G_{\la-1}(\rho,v)\big(&4\la kk'a_1'+4\la^2k'^2a_1+2\la kk''a_1+2\la\eps A_0 kk'a_1-2\la\eps B_0k^2a_1\big)\\
 +\, G_\la(\rho,v)\big(&a_1''+4(\la+1)kk'a_2'+4(\la+1)^2 k'^2a_2+2(\la+1)kk''a_2\\
&+\eps A_0a_1'+2\la\eps B_0a_1+2(\la+1)\eps A_0kk'a_2-2(\la+1)\eps B_0k^2a_2\\
& +2\la\eps^2 A_1kk'a_1-2\la\eps^2 B_1k^2a_1\big)\\
+\, G_{\la+1}(\rho,v&)\big(a_2''+2(\la+1)\eps B_0a_2+\eps(A-A_0)a_1'+2\la\eps^3A_2kk'a_1+\eps A_0a_2'\\
&\hspace{3.5mm}+2(\la+1)\eps(A-A_0)kk'a_2+2(\la+1)\eps B_0a_2+2\la\eps^2 B_1a_1\\
&\hspace{3.5mm}-2\la\eps^3 B_2v^2a_1-2(\la+1)\eps(B-B_0)v^2a_2\big)\\
 +\, g_{\rho\rho}-k'(\rho&)^2g_{vv}+\eps Ag_\rho+\eps Bvg_v.
\eeqas
Thus, we see that, in order to cancel the highest order singularities, $a_1$ must solve
\beqs
\frac{a_1'}{a_1} = -\frac{k''}{2k'} - \lambda \frac{k'}{k} - \frac{\eps}{2}A_0 + \frac{\eps}{2}\frac{k}{k'}B_0,
\eeqs
so that
\begin{equation}\label{def:a1}
a_1(\rho) = c_{*,\lambda}k(\rho)^{-\lambda}k'(\rho)^{-\frac{1}{2}}e^{\tilde{a}(\rho)},
\end{equation}
where the constant $c_{*,\lambda}>0$ is determined to satisfy the initial conditions and
\begin{equation}\label{4.1a}
\tilde{a}(\rho) := \frac{\eps}{2}\int_{0}^{\rho}\big(-A_0(s) + \frac{k(s)}{k'(s)}B_0(s)\big)\dd s,
\end{equation}
as defined in \eqref{3.14a}.
Defining $\alpha_1(\rho) := a_1(\rho)k(\rho)^{2\lambda + 1}$, we obtain
\beq\label{eq:al1-ode}
\frac{\alpha_1'}{\alpha_1} = -\frac{k''}{2k'} + (\lambda +1)\frac{k'}{k} - \frac{\eps}{2}A_0 + \frac{\eps}{2}\frac{k}{k'}B_0.
\eeq
Cancelling the next highest order singularities, we obtain the following equation for $a_2(\rho)$:
\beqs
a_2'+a_2\Big(\frac{k''}{2k'}+(\la+1)\frac{k'}{k}+\frac{\eps}{2}A_0-\frac{\eps}{2}\frac{k}{k'}B_0\Big)
=-\frac{1}{4(\la+1)kk'}\tilde{W},
\eeqs
where
\beq \label{eq:W}
\tilde{W}=a_1'' + \eps A_0a_1' + 2\lambda \eps^2a_1(kk'A_1 - k^2B_1)+2\lambda\eps B_0 a_1.
\eeq
Defining $\alpha_2(\rho) := a_2(\rho)k(\rho)^{2\lambda + 3}$, we have
\beqa\label{eq:al2-ode}
&\alpha_2' + \alpha_2\big(\frac{k''}{2k'} - (\lambda + 2)\frac{k'}{k} + \frac{\eps}{2}A_0 - \frac{\eps}{2}\frac{k}{k'}B_0\big) \\
&= -\frac{k}{4(\lambda +1)k'}\big(\alpha_1''+\eps A_0\alpha_1'-\eps B_0\alpha_1+2\lambda \eps^2\alpha_1(kk'A_1 - k^2B_1)\big)=: \Omega.
\eeqa
We take the less singular solution to this singular differential equation given by
\begin{equation*}
\alpha_2(\rho) = e^{\tilde{a}(\rho)}k(\rho)^{\lambda + 2}k'(\rho)^{-\frac{1}{2}}
\int_{0}^{\rho}e^{-\tilde{a}(\tau)}k(\tau)^{-\lambda - 2}k'(\tau)^{\frac{1}{2}}\Omega(\tau)\,\dd \tau.
\end{equation*}
From the observations in Lemma \ref{lemma:k-asymptotics},
we see that $k(\tau)^{-\lambda - 2}k'(\tau)^{\frac{1}{2}}\Omega(\tau) = O(\tau^{\frac{\gamma -3}{2}})$ as $\tau\to0$,
so that it is integrable.
Then we conclude that
\begin{equation}\label{eq:a_2coefficient}
\begin{split}
a_2(\rho) =&\, e^{\tilde{a}(\rho)}k(\rho)^{-\lambda - 1}k'(\rho)^{-\frac{1}{2}}
  \int_{0}^{\rho}e^{-\tilde{a}(\tau)}k(\tau)^{-\lambda - 2}k'(\tau)^{\frac{1}{2}}\Omega(\tau)\,\dd \tau\\
=&-\frac{1}{4(\lambda+1)}e^{\tilde{a}(\rho)}k(\rho)^{-\lambda-1}k'(\rho)^{-\frac{1}{2}}
  \int_{0}^{\rho}e^{-\tilde{a}(\tau)}k(\tau)^{\lambda}k'(\tau)^{-\frac{1}{2}}\tilde{W}(\tau)\,\dd \tau
\end{split}
\end{equation}
is well defined, where $\tilde{W}$ is defined in \eqref{eq:W}.
Throughout the paper, we exploit the fact that there exist $C_1$ and $C_2$
depending on $\rho_M$ such that $0<C_1\leq e^{\tilde{a}(\rho)}\leq C_2<\infty$.

\subsection{Proof of the existence and regularity of the entropy kernel}\label{sec:entropykernelproofs}
As coefficients $A$ and $B$ of the first-order terms in the entropy equation depend on both $\rho$ and $v$,
we isolate the principal part of the operator in order to exploit the Fourier transform.
For this purpose, we define an operator $\tilde{\mathbf{L}}$ by
\begin{equation}\label{def:P}
\tilde{\mathbf{L}} := \partial_{\rho \rho} - k'(\rho)^2\partial_{vv} + \eps A_0(\rho)\partial_{\rho} - \beta(\rho),
\end{equation}
where
\begin{equation*}
\beta(\rho) := \frac{\alpha_\sharp''(\rho)}{\alpha_\sharp(\rho)} + \eps A_0(\rho)\frac{\alpha_\sharp'(\rho)}{\alpha_\sharp(\rho)},
\end{equation*}
and
$$
\alpha_\sharp(\rho)=c_\sharp e^{-\frac{\eps}{2}\int_0^{\rho}A_0(s)\,\dd s}k(\rho)^{\lambda+1}k'(\rho)^{-\frac{1}{2}},
$$
with constant $c_\sharp>0$ to be chosen later.
We observe that $\al_\sharp$ satisfies the equation:
$$
\frac{\alpha_\sharp'}{\alpha_\sharp} = -\frac{k''}{2k'} + \frac{(\lambda +1)k'}{k} - \frac{\eps}{2}A_0.
$$
By the asymptotics for $k(\rho)$ given in Lemma \ref{lemma:k-asymptotics},
we find  that $\beta(\rho)$ is $O(\rho^{\gamma - 3})$ as $\rho \rightarrow 0$.

With operator $\tilde{\mathbf{L}}$, we work in the Fourier space.
We therefore determine an equation for the expression: $\mathcal{F}(\tilde{\mathbf{L}}g)(\rho,\xi)$,
where $\F$ denotes the Fourier transform in variable $v$.
To calculate the Fourier transform of function $[k(\rho)^2-v^2]_+^\la$, we use the following facts:

For ease of notation, we write $f_{\lambda}(y):=[1-y^2]_+^{\lambda}$ so that
$$
[k(\rho)^2-v^2]_+^\la=k(\rho)^{2\la}f_\la\big(\frac{v}{k(\rho)}\big).
$$
Recalling now from \cite{GelfandShilov} that the Fourier transform of $f_\la(y)$ is
$$
\widehat{f}_\la(\xi)=\sqrt{\pi}\Ga(\la+1)2^{\la+\half}|\xi|^{-\la-\half}J_{\la+\half}(\xi),
$$
where $J_\nu$ is the Bessel function of first type of order $\nu$, we have
$$
\F\big([k(\rho)^2-v^2]_+^\la\big)=k(\rho)^{2\la+1}\widehat{f}_{\la}(k(\rho)\xi)
=C\big(\frac{k(\rho)}{\xi}\big)^{\la+\half}J_{\la+\half}(k(\rho)\xi).
$$
We are therefore able to derive an equation for $\tilde{\mathbf{L}}g$ in the Fourier space.

\begin{prop}\label{prop:Pgexpression}
The remainder function $g(\rho,v)$ satisfies
 \begin{align*}
  \F(\tilde{\mathbf{L}}g)(\rho,\xi)=&\,\mathcal{F}({\mathcal{S}(g)})(\rho,\xi)\\
  :=&-\mathcal{F}(\eps(A-A_0)g_\rho)-\mathcal{F}(\eps Bvg_v)-\beta(\rho)\widehat{g}+H_0(\rho)\widehat{f}_{\lambda+1}(k(\rho)\xi) +\eps^2r(\rho,\xi),
\end{align*}
where $H_0(\rho)=O(\rho^{-1+2\theta})$ as $\rho\rightarrow0$,
and $r(\rho,\xi)=O(\rho^{-1+2\th}(k(\rho)|\xi|)^{-\la-1-\al-\half})$ as $|\xi|\to\infty$, for some $\al>0$.
In particular, $r(\rho,\xi)$ is asymptotically like $\widehat{f}_{\la+1+\al}(k(\rho)\xi)$ as $\xi\to\infty$.
\end{prop}

\begin{proof}
We write $X^1:= a_1(\rho)[k(\rho)^2-v^2]_+^\la$ and $X^2:=a_2(\rho)[k(\rho)^2-v^2]_+^{\la+1}$.
From the considerations above, we find that $\F(X^1)=\alpha_1(\rho)\widehat{f}_{\lambda}(k(\rho)\xi)$
and $\F(X^2)=\alpha_2(\rho)\widehat{f}_{\lambda+1}(k(\rho)\xi)$.
Now, rearranging the entropy equation, we obtain
\begin{align*}
  \F(\tilde{\mathbf{L}}g)(\rho,\xi)=&-I-II-\eps^3\mathcal{F}(A_2[k^2-v^2]_+^2X^1_{\rho})-\eps\mathcal{F}((A-A_0)X^2_{\rho})\\
  &-\mathcal{F}(\eps^3 B_2[k^2-v^2]_+^2vX^1_v+\eps(B-B_0)vX^2_v)\\
  &-\mathcal{F}(\eps(A-A_0)g_\rho)-\mathcal{F}(\eps Bvg_v)-\beta(\rho)\widehat{g},
\end{align*}
where
\beqas
   I:=&\,\F(X^1_{\rho\rho})+k'(\rho)^2\xi^2\F(X^1) + \mathcal{F}(\eps (A_0+\eps A_1[k^2-v^2]_+)X^1_{\rho})\\
   &+\mathcal{F}(\eps(B_0v + \eps B_1[k^2-v^2]_+v)X^1_v)
\eeqas
and
\beqas
II:=&\,\F(X^2_{\rho\rho})+k'(\rho)^2\xi^2\F(X^2) + \eps A_0\F(X^2_{\rho}) + \eps\mathcal{F}(B_0vX^2_v).
\eeqas
As $\widehat{f}_{\lambda}'(y)=-\frac{2\lambda+1}{y}\widehat{f}_\lambda + \frac{2\lambda}{y}\widehat{f}_{\lambda-1}$, we calculate
\beqas
  I=&\,\widehat{f}_\lambda\big(\alpha_1''+\eps A_0\alpha_1'-\eps B_0\alpha_1+2\lambda\eps^2\alpha_1(kk'A_1-k^2B_1)\big)\\
  &+\widehat{f}_{\lambda+1}\big(\eps^2A_1\alpha_1'k^{2\lambda+3}+2\lambda\eps^2 B_1k^2\alpha_1\big).
\eeqas
 Similarly,
\begin{align*}
  II=&-\widehat{f}_{\lambda}\big(\alpha_1''+\eps A_0\alpha_1'-\eps B_0\alpha_1+2\lambda\eps^2\alpha_1(kk'A_1-k^2B_1)\big)\\
  &+\widehat{f}_{\lambda+1}\Big(\alpha_2''+\eps A_0\alpha_2'-\eps B_0\alpha_2\\
  &\qquad\qquad +\frac{2\lambda+3}{2(\lambda+1)}\big(\alpha_1''+\eps A_0\alpha_1'-\eps B_0\alpha_1+2\lambda\eps^2\alpha_1(kk'A_1-k^2B_1)\big)\Big),
\end{align*}
where we have used that coefficients  $\al_1$ and $\al_2$ satisfy \eqref{eq:al1-ode} and \eqref{eq:al2-ode}.
Then we obtain
\begin{equation*}
 \begin{split}
  I+II=\widehat{f}_{\lambda+1}\Big(&\alpha_2''+\eps A_0\alpha_2'-\eps B_0\alpha_2+\eps^2A_1\alpha_1'k^{2\lambda+3}+2\lambda\eps^2 B_1k^2\alpha_1\\
  &+\frac{2\lambda+3}{2(\lambda+1)}\big(\alpha_1''+\eps A_0\alpha_1'-(2\lambda+1)\eps B_0\alpha_1+2\lambda\eps^2\alpha_1(kk'A_1-k^2B_1)\big)\Big).
 \end{split}
\end{equation*}
Denote
$$
\eps^2r(\rho,\xi):=\F\big(-\eps^3(A_2[k^2-v^2]_+^2X^1_{\rho})-\eps(A-A_0)X^2_{\rho}-\eps^3 B_2[k^2-v^2]_+^2vX^1_v+\eps(B-B_0)X^2_v\big)
$$
as a higher order term.
Then we have
\begin{align*}
  \F(\tilde{\mathbf{L}}g)(\rho,\xi)=\mathcal{F}({\mathcal{S}(g)})\,\,&\\
  :=-\mathcal{F}(\eps(A&-A_0)g_\rho)-\mathcal{F}(\eps Bvg_v)-\beta(\rho)\widehat{g}-H_0(\rho)\widehat{f}_{\lambda+1}(k(\rho)\xi) +\eps^2r(\rho,\xi),
\end{align*}
where
\beqas
H_0(\rho)=&\,\alpha_2''+\eps A_0\alpha_2'-\eps B_0\alpha_2+\eps^2A_1k^{2\lambda+3}\alpha_1' +2\lambda \eps^2 B_1k^2\alpha_1\\
  \qquad&+\frac{2\lambda+3}{2(\lambda+1)}\big(\alpha_1''+\eps A_0\alpha_1'-(2\lambda+1)\eps B_0\alpha_1+2\lambda\eps^2 (kk'A_1-k^2B_1)\alpha_1\big)
\eeqas
as required.
We note that $H_0(\rho)=O(\rho^{-1+2\theta})$ as $\rho\rightarrow0$ from the limiting forms of $k(\rho)$ given
in Lemma \ref{lemma:k-asymptotics}.
Moreover, one can check that the remainder function $r(\rho,\xi)$ acts as the Fourier transform of the product of a smooth (even Schwartz)
function with $f_{\la+1+\al}$ for some $\al>0$,
so that
$$
r(\rho,\xi)=O\big(\rho^{-1+2\th}(k(\rho)|\xi|)^{-\la-1-\al-\half}\big)\qquad\, \text{ as }\rho\to0.
$$
\end{proof}

Recalling the definition of $\tilde{\mathbf{L}}$ and applying the Fourier transform with respect to $v$,
we obtain the following differential equation:
\beq\label{4.8a}
\F(\tilde{\mathbf{L}}g)(\rho,\xi)
:=\widehat{g}_{\rho\rho}(\rho,\xi)+k'(\rho)^2\xi^2\widehat{g}(\rho,\xi)+\eps A_0(\rho)\widehat{g}_\rho(\rho,\xi)-\beta(\rho)\widehat{g}(\rho,\xi)
=\mathcal{F}({\mathcal{S}(g)}).
\eeq
We solve this equation by the method of variation of parameters.
In order to do this, we require the fundamental solutions of operator $\F(\tilde{\mathbf{L}}\,\cdot)$.
These fundamental solutions are, by definition, $\widehat{\chi}^{\sharp}$ and $\widehat{\chi}^{\flat}$ determined by
\begin{equation*}
\begin{cases}
\F\big(\tilde{\mathbf{L}}\chi^{\sharp}\big) = 0,\\
\widehat{\chi}^{\sharp}\rvert_{\rho=0} = 0,\,\,\,\widehat{\chi}^{\sharp}_{\rho}\rvert_{\rho = 0} = 1,
\end{cases}
\end{equation*}
and
\begin{equation*}
\begin{cases}
\F\big({\tilde{\mathbf{L}}\chi^{\flat}}\big) = 0, \\
\widehat{\chi}^{\flat}\rvert_{\rho=0} = 1,\,\,\, \widehat{\chi}^{\flat}_{\rho}\rvert_{\rho = 0} = 0,
\end{cases}
\end{equation*}
respectively.

\begin{lemma}\label{lemma:Ltildefundamental}
The fundamental solutions of the Fourier transformed equations are
\begin{equation}
\widehat{\chi}^{\sharp}(\rho,\xi) = \alpha_\sharp(\rho)(\xi k(\rho))^{-\nu}J_{\nu}(\xi k(\rho)), \quad\,
\widehat{\chi}^{\flat}(\rho,\xi) = \alpha_\sharp(\rho)\big(\frac{k(\rho)}{\xi}\big)^{-\nu}Y_{\nu}(\xi k(\rho)),
\end{equation}
where $\nu=\lambda+\frac{1}{2}$, $J_{\nu}$ and $Y_{\nu}$ are the Bessel functions of order $\nu$ of first and second type respectively,
and the constant $c_\sharp>0$ in the definition of $\alpha_\sharp(\rho)$ is chosen to satisfy the initial conditions.
\end{lemma}

\begin{proof}
This follows by a direct calculation from the identities:
\begin{align*}
 \mathcal{C}_{\nu}'(y)=\mathcal{C}_{\nu-1}(y)-\frac{\nu}{y}\mathcal{C}_\nu(y)=
 \mathcal{C}_{\nu}'(y)=-\mathcal{C}_{\nu+1}(y)+\frac{\nu}{y}\mathcal{C}_\nu(y)
\end{align*}
for $\mathcal{C}_\nu(y)=J_\nu(y),Y_\nu(y)$;
also see \cite{NIST}.
\end{proof}

We solve the differential equation \eqref{4.8a}
by the method of variation of parameters.
First, we calculate the Wronskian of the fundamental solutions.
Note  that  the Wronskian ({\it cf.} \cite{NIST}):
\begin{equation*}
  w(s,\xi):=Y_\nu(\xi k(s))J_\nu'(\xi k(s))-J_\nu(\xi k(s))Y_\nu'(\xi k(s))=\frac{2\nu}{\xi k(s)},
\end{equation*}
so that
\begin{align*}
  W(s,\xi):=\big(\widehat{\chi}^\sharp(\xi k(s))\partial_s\widehat{\chi}^{\flat}(\xi k(s))-\widehat{\chi}^{\flat}(\xi k(s))\partial_s\widehat{\chi}^\sharp(\xi k(s))\big)
  =c_\sharp^2e^{-\eps\int_0^{s}A_0(\tau)\dd\tau}.
\end{align*}
The method of variation of parameters then gives that a particular solution
of \eqref{4.8a} is
\beqa\label{eq:ghatrep}
\widehat{g}(\rho,\xi) =& \int_{0}^{\rho}\frac{\widehat{\chi}^{\sharp}(\rho,\xi)\widehat{\chi}^{\flat}(s,\xi) - \widehat{\chi}^{\sharp}(s,\xi)\widehat{\chi}^{\flat}(\rho,\xi)}{\widehat{\chi}^{\sharp}_s(s,\xi)\widehat{\chi}^{\flat}(s,\xi) - \widehat{\chi}^{\sharp}(s,\xi)\widehat{\chi}^{\flat}_s(s,\xi)}\F(\tilde{\mathbf{L}}g)(s,\xi)\,\dd s\\[2mm]
=&\Big(\frac{k(\rho)}{k'(\rho)}\Big)^{\frac{1}{2}}e^{-\frac{\eps}{2}\int_0^{\rho}A_0(\tau)\,\dd \tau}
  \int_0^{\rho}K(\rho,s;\xi)e^{\frac{\eps}{2}\int_0^{s}A_0(\tau)\,\dd \tau}\Big(\frac{k(s)}{k'(s)}\Big)^{\frac{1}{2}}\F(\mathcal{S}(g))(s,\xi)\,\dd s,
\eeqa
where
$$
K(\rho,s;\xi):=Y_\nu(\xi k(\rho))J_\nu(\xi k(s))-J_\nu(\xi k(\rho))Y_\nu(\xi k(s)).
$$

Define a new integral kernel $\tilde{K}(\rho,s;\xi)$  by
\beq\label{eq:Ktilde}
\tilde{K}(\rho,s;\xi)=\Big(\frac{k(\rho)}{k'(\rho)}\Big)^{\frac{1}{2}}\Big(\frac{k(s)}{k'(s)}\Big)^{\frac{1}{2}}
  K(\rho,s;\xi)e^{-\frac{\eps}{2}\int_s^{\rho}A_0(\tau)\,\dd \tau},
  \eeq
and look for a fixed point of
  \beqs
  \widehat{g}(\rho,\xi)=\int_0^\rho\tilde{K}(\rho,s,\xi)\F(\mathcal{S}(g))\,\dd s.
  \eeqs
Then we show via a fixed point argument that such a function $\widehat{g}(\rho,\xi)$ exists in Theorem \ref{thm:entkernelexistence}.

Before proving this theorem, we first make a few observations.
To simplify the notation and bounds later, we set
\begin{align*}
 Q_{\pm\nu}(y):=\begin{cases}
            |y|^{\pm\nu} &\text{ for }|y|\leq1,\\
            |y|^{-\frac{1}{2}} &\text{ for }|y|\geq1,
           \end{cases}
\end{align*}
and
\begin{align*}
  R(y):=\begin{cases}
            1 &\text{ for }|y|\leq1,\\
            |y|^{-1} &\text{ for }|y|\geq1.
           \end{cases}
\end{align*}
Then $|J_\nu(y)|\leq CQ_\nu(y)$ and $|Y_\nu(y)|\leq CQ_{-\nu}(y)$ for $y>0$.
Thus, we may bound
\begin{equation}\label{ineq:K}
|K(\rho,s;\xi)|\leq CQ_\nu(\xi k(\rho))Q_\nu(\xi k(s))^{-1}R(\xi k(s)),
\end{equation}
where $C>0$ is independent of $(\rho, s, \xi)$ for $0\leq s\leq\rho\leq\rho_M$ and $\xi\in\R$.

The following lemma provides accurate estimates for various $L^p$ and weighted $L^p$ norms
of the kernel $\tilde{K}(\rho,s;\xi)$. These are simple consequences of \eqref{ineq:K}.

\begin{lemma}\label{lemma:Kbounds}
For $0\leq s\leq\rho\leq\rho_M$ and $\xi\in\R$,
\beqas
&\tilde{K}(\rho,\rho;\xi)=0,\qquad \tilde{K}(\rho,0;\xi)=0.\\
&\|\tilde{K}(\rho,s;\cdot)\|_{L^2}\leq\, C\rho^{1-\frac{\theta}{2}},
\eeqas
and, for $0\leq\mu<\frac{1}{2}$,
\begin{align}
&\|\xi\tilde{K}(\rho,s;\xi)\|_{L^{\infty}_{\xi}}\leq\, C\rho^{1-\theta},\\
&\||\xi|^{\mu}\tilde{K}(\rho,s;\xi)\|_{L^2_\xi}\leq \,C\rho^{1-\frac{\theta}{2}-\mu\theta}.
\end{align}
\end{lemma}
With these bounds in hand, we are now in a position to state and prove the main theorem of this section.
This theorem gives the existence of the remainder function $\widehat{g}(\rho,\xi)$,
and hence the existence of the entropy kernel itself.

\begin{thm}[Existence of the Entropy Kernel]\label{thm:entkernelexistence}
For all $\gamma \in (1,3)$, there exists
$$
\widehat{g} \in L^{\infty}(0,\rho_{M};L^2(\mathbb{R}))
$$
that is a fixed point of
 \beq\label{eq:gfixedpoint}
  \widehat{g}(\rho,\xi)=\int_0^\rho\tilde{K}(\rho,s,\xi)\F(\mathcal{S}(g))(s,\xi)\,\dd s.
  \eeq
  The remainder function $\widehat{g}(\rho,\xi)$ satisfies
  \beq\label{ineq:ghatL2bound}
  \|\widehat{g}(\rho,\cdot)\|_{L^2}+\|\rho\widehat{g}_\rho(\rho,\cdot)\|_{L^2}\leq C\rho^{1+\frac{3\theta}{2}}e^{\frac{\rho^{2\theta}}{2\theta}}.
  \eeq
\end{thm}

\begin{proof}
To establish the existence of the kernel, we argue with a constructive fixed point scheme.
Given the $n$th approximation $\widehat{g}^n(\rho,\xi)$,
we construct $\widehat{g}^{n+1}(\rho,\xi)$  by setting
\beqas
 \widehat{g}^{n+1}(\rho,v) :=& \int_0^{\rho}\tilde{K}(\rho,s;\xi)\F\big(\mathcal{S}(g^n)\big)\,\dd s\\
 =&\int_0^{\rho}\tilde{K}(\rho,s;\xi)\Big(H_0(s)\widehat{f}_{\lambda+1}(k(s)\xi)+\eps^2 r(s,\xi)\\
 &\qquad\qquad\quad\quad\,\,\, -\mathcal{F}(\eps(A-A_0)g^n_s-\eps Bvg^n_v-\beta(s)g^n)\Big)\,\dd s,
\eeqas
and begin the procedure with $\widehat{g}^0(\rho,\xi) = 0$.
To show that this scheme converges, we estimate the increment: $\widehat{g}^{n+1}-\widehat{g}^n$.
Set $\widehat{G}^{n+1} := \widehat{g}^{n+1}-\widehat{g}^{n}$.
By linearity of the Fourier transform, we have
\beqas
\widehat{G}^{n+1}(\rho,\xi)=\int_0^\rho\tilde{K}(\rho,s;\xi)\Big(-\mathcal{F}(\eps(A-A_0)G^n_s-\eps BvG^n_v-\beta(s)G^n)\Big)\,\dd s.
\eeqas
We first apply integration by parts (recalling from Lemma \ref{lemma:Kbounds} that $\tilde{K}(\rho,0;\xi)=\tilde{K}(\rho,\rho;\xi)=0$) to obtain
\beqa\label{eq:antelope}
&\widehat{G}^{n+1}(\rho,\xi)\\
  &=\int_0^{\rho}\tilde{K}(\rho,s;\xi)\Big(\F(-\eps(A-A_0)G_s^n(s,v))-\F(\eps B(s,v)vG^n_v(s,v))-\beta(s)\widehat{G}^n(s,v)\Big)\,\dd s\\
  &=\int_0^{\rho}\Big(\tilde{K}(\rho,s;\xi)\big(\F(\eps(A-A_0)_sG^n-\eps(BvG^n)_v+\eps(Bv)_vG^n-\beta(s){G}^n(s,v)\big)(s,\xi)\big)\\
  & \qquad\quad\,\,\,+\tilde{K}_s(\rho,s;\xi)\F(\eps(A-A_0)G^n)(s,\xi)\Big)\,\dd s.
\eeqa
We recall that the Fourier transform of a product is the convolution of the Fourier transforms.
We therefore use the bounds of Lemma \ref{lemma:Kbounds} and Young's inequality for convolutions to estimate, for a typical term,
\beqas
&\Big\|\int_0^{\rho}\tilde{K}(\rho,s;\xi)\F(\eps(A-A_0)_sG^n)\,\dd s\Big\|_{L^2_\xi}\\
&\leq \int_0^\rho\|\tilde{K}(\rho,s;\cdot)\|_{L^2}\|\F(\eps(A-A_0)_s)* \F(G^n)\|_{L^\infty_\xi}\,\dd s\\
&\leq \int_0^{\rho}\|\tilde{K}(\rho,s;\cdot)\|_{L^2}\|\eps\F((A-A_0)_s)(s,\cdot)\|_{L^2}\|\widehat{G}^n(s,\cdot)\|_{L^2}\,\dd s\\
&\leq \,C\eps^2\int_0^\rho \rho^{1-\frac{\th}{2}}s^{-2+\frac{5\th}{2}}\|\widehat{G}^n(s,\cdot)\|_{L^2}\,\dd s,
\eeqas
where we have estimated the $L^2$ norm of $\F((A-A_0)_s)$ by applying Plancherel's theorem and estimating
with the bounds of Lemma \ref{lemma:A,Bdecomposition} and the compact support of the function to see
\beqas
\|\eps(A-A_0)_s\|^2_{L^2_v}\leq C\int_{-k(s)}^{k(s)}\eps^4 s^{2\ga-6}\,\dd v
\leq
C\eps^4 s^{5\th-4}.
\eeqas
Treating the other terms of \eqref{eq:antelope} similarly, we find
\beqas
  \|\widehat{G}^{n+1}(\rho,\cdot)\|_{L^2}
  \leq C\int_0^{\rho}\Big(&\|\tilde{K}(\rho,s;\cdot)\|_{L^2}\|\F(\eps(A-A_0)_s)(s,\cdot)\|_{L^2}\|\widehat{G}^n(s,\cdot)\|_{L^2}\\
  &+\|\tilde{K}_s(\rho,s;\cdot)\|_{L^2}\|\F(\eps(A-A_0))(s,\cdot)\|_{L^2}\|\widehat{G}^n(s,\cdot)\|_{L^2}\\
  &+\|\xi\tilde{K}(\rho,s;\xi)\|_{L^{\infty}_\xi}\|\F(\eps BvG^n)(s,\cdot)\|_{L^2}\\
  &+\|\tilde{K}(\rho,s;\cdot)\|_{L^{\infty}}\|\F(\eps\partial_v(Bv)G^n)(s,\cdot)\|_{L^2}\\
  &+\|\tilde{K}(\rho,s;\cdot)\|_{L^{\infty}}|\beta(s)|\|\widehat{G}^n(s,\cdot)\|_{L^2}\Big)\,\dd s\\
  \leq C\int_0^{\rho}\big(&\eps^2\rho^{1-\frac{\theta}{2}}s^{-2+\frac{5\theta}{2}}+\eps\rho^{1-\theta}s^{-2+3\theta}+\rho s^{-2+2\theta}\big)\|\widehat{G}^n(s,\cdot)\|_{L^2}\,\dd s.
\eeqas
In particular, we have obtained
\begin{equation}\label{ineq:Gn+1}
 \rho^{-1}\|\widehat{G}^{n+1}(\rho,\cdot)\|_{L^2}\leq C\int_0^{\rho}s^{-1+2\theta}s^{-1}\|\widehat{G}^n(s,\cdot)\|_{L^2}\,\dd s,
\end{equation}
where $C>0$ is independent of $\rho$.

Before estimating $\widehat{G}^1(\rho,\xi)$, we note that, for $0<s\leq\rho$,
\begin{align}
  &\Big\|Q_\nu(\xi k(\rho))Q_\nu(\xi k(s))^{-1}R(\xi k(s))\Big|\frac{J_{\nu+1}(\xi k(s))}{(\xi k(s))^{\nu+1}}\Big|\,\Big\|_{L^2_\xi}\nonumber\\
 &\leq C\Big(\int_0^{\frac{1}{k(\rho)}}\Big(\frac{k(\rho)}{k(s)}\Big)^{2\nu}\,\dd \xi
   +\int_{\frac{1}{k(\rho)}}^{\frac{1}{k(s)}}k(\rho)^{-1}k(s)^{-2\nu}|\xi|^{-2\nu-1}\,\dd \xi\nonumber\\
 &\qquad\,\,\,\,+\int_{\frac{1}{k(s)}}^{\infty}k(\rho)^{-1}k(s)^{-2\nu-4}|\xi|^{-2\nu-5}\,\dd \xi\Big)^{\frac{1}{2}}\nonumber\\
 &\leq C\rho^{\frac{1-\theta}{2}}s^{-\frac{1}{2}}.\label{ineq:armadillo}
 \end{align}
It follows from Proposition \ref{prop:Pgexpression} that
$$
\|\tilde{K}(\rho,s;\xi)r(s,\xi)\|_{L^2_\xi}
\leq C\|\tilde{K}(\rho,s;\xi)H_0(s)\widehat{f}_{\lambda+1}(k(s)\xi)\|_{L^2_\xi}.
$$
Then we use $|H_0(\rho)|\leq C\rho^{-1+2\th}$ from Proposition \ref{prop:Pgexpression} and employ estimate \eqref{ineq:armadillo}
to calculate
\begin{align}
  \|\widehat{G}^1(\rho,\cdot)\|_{L^2}
  &\leq C\Big\|\Big(\frac{k(\rho)}{k'(\rho)}\Big)^{\frac{1}{2}}\int_0^{\rho}K(\rho,s;\xi)
    \Big(\frac{k(s)}{k'(s)}\Big)^{\frac{1}{2}}H_0(s)\frac{J_{\nu+1}(\xi k(s))}{(\xi k(s))^{\nu+1}}\,\dd s\Big\|_{L^2_\xi}\nonumber\\
  &\leq C\rho^{\frac{1}{2}}\int_0^{\rho}\Big\|Q_\nu(\xi k(\rho))Q_\nu(\xi k(s))^{-1}R(\xi k(s))s^{-\frac{1}{2}+2\theta}
     \Big|\frac{J_{\nu+1}(\xi k(s))}{(\xi k(s))^{\nu+1}}\Big|\,\Big\|_{L^2_\xi}\,\dd s\nonumber\\
  &\leq C\rho^{1-\frac{\th}{2}}\int_0^{\rho}s^{-1+2\theta}\,\dd s\nonumber\\
  &\leq C \rho^{1+\frac{3\theta}{2}}. \label{ineq:G1}
\end{align}
Thus, combining \eqref{ineq:Gn+1} and \eqref{ineq:G1}, we may bound iteratively
\beqas
 \rho^{-1}\|\widehat{G}^{n+1}(\rho,\xi)\|_{L^2_\xi} \leq& \int_{0}^{\rho}s_1^{-1 + 2\theta}s_1^{-1}\|\widehat{G}^n(s_1,\cdot)\|_{L^2}\,\dd s_1\\
 \leq&\, C\int_{0}^{\rho}\int_{0}^{s_1}\cdots\int_{0}^{s_{n-1}}(s_n\cdots s_1)^{-1 + 2\theta}s_n^{-1}\|\widehat{G}^1(s_n,\cdot)\|_{L^2}\,\dd s_n\cdots \dd s_1\\
 \leq& \,C \rho^{\frac{3\theta}{2}+2n\theta}\prod_{j=1}^{n}\frac{1}{\frac{3\theta}{2}+2j\theta}\\
 \leq&\, C\frac{(2\theta)^{-n}}{n!}\rho^{\frac{3\theta}{2}+2n\theta}.
\eeqas
Hence, we obtain
$$
\|\widehat{G}^{n+1}(\rho,\xi)\|_{L^2_\xi}\leq C\frac{(2\theta)^{-n}}{n!}\rho^{1+\frac{3\theta}{2}+2n\theta}.
$$
In particular,
$\widehat{g}(\rho,\xi):=\lim_{n\rightarrow\infty}\widehat{g}^n(\rho,\xi)=\lim_{n\rightarrow\infty}\sum_{k=1}^n\widehat{G}^k(\rho,\xi)$
exists and satisfies
$$
\|\widehat{g}(\rho,\cdot)\|_{L^2}\leq C\sum_{n=0}^{\infty}\frac{(2\theta)^{-n}}{n!}\rho^{1+\frac{3\theta}{2}+2n\theta}=C\rho^{1+\frac{3\theta}{2}}e^{\frac{\rho^{2\theta}}{2\theta}}.
$$
Moreover, $\widehat{g}$ is the desired fixed point.
In an analogous argument, we can obtain the estimate for $\rho \hat{g}_\rho$.
This completes the proof.
\end{proof}

\begin{rmk}\label{rmk:xigxi}
 A totally analogous argument yields that $\xi\widehat{g}_\xi\in L^2_\xi$ with
 $$
 \|\xi\widehat{g}_\xi\|_{L^2_\xi}\leq C\rho^{1+\frac{3\theta}{2}}e^{\frac{\rho_M^{2\theta}}{2\theta}}.
 $$
\end{rmk}

\begin{thm}[Regularity of $g(\rho,v)$]\label{thm:entropykernelregularity}
 The remainder function $g=g(\rho,v)$ is such that $\partial_v^\mu g$ is H\"{o}lder continuous in $(\rho,v)$ for $\rho>0$ for all $\mu$ with $0\leq\mu<\lambda+2$. In addition, if $0<\be<\mu$,
 $$
 |\partial_v^\be g(\rho,v)|\leq C\rho^{1+(1-2\mu+\beta)\th}[k(\rho)^2-v^2]_+^{\mu-\be}
\qquad \mbox{for all $\rho$ with  $0\leq\rho\leq\rho_M$ and $v\in\R$}.
 $$
\end{thm}

\begin{proof}
Recalling the definition of fractional derivatives in \eqref{def:fractionalderivative},
we see that the fractional derivative of order $\mu$ of the remainder function $g(\rho,v)$
may be bounded by
$$
|\partial_v^\mu g(\rho,v)|\leq C\int_{\mathbb{R}}|\xi|^\mu|\widehat{g}(\rho,\xi)|\,\dd\xi.
$$
For $0\leq\mu<\frac{1}{2}$, we apply the same method as in the proof of Theorem \ref{thm:entkernelexistence}
with the bounds of Lemma \ref{lemma:Kbounds} to obtain
 \begin{align}
   \int_{\mathbb{R}}&|\xi|^\mu|\widehat{g}(\rho,\xi)|\,\dd\xi\nonumber\\
   \leq&\, C\int_{\mathbb{R}}|\xi|^\mu\int_0^\rho\Big|\tilde{K}(\rho,s;\xi)\Big(H_0(s)\Big|\frac{J_{\nu+1}(\xi k(s))}{(\xi k(s))^{\nu+1}}\Big|-\F\big(\eps(A-A_0)g_s\hspace{-0.2mm}-\hspace{-0.2mm}\eps Bvg_v\hspace{-0.2mm}-\hspace{-0.2mm}\beta(s)g\big)\Big)\Big|\,\dd s \dd\xi\nonumber \\
   \leq &\,C\int_0^{\rho}\big\||\xi|^{\mu}\tilde{K}(\rho,s;\xi)\big\|_{L^2_\xi}\Big(\Big\|H_0(s)\Big|\frac{J_{\nu+1}(\xi k(s))}{(\xi k(s))^{\nu+1}}\Big|\,\Big\|_{L^2_\xi}+\eps\|(A-A_0)_s\|_{L^{\infty}_v}\|\widehat{g}\|_{L^2_\xi} \nonumber \\
    &\hspace{38mm}+ |\beta(s)|\|\widehat{g}\|_{L^2_\xi}+\eps\|B(s,v)\|_{L^{\infty}_v}\big(\|\widehat{g}\|_{L^2_\xi}+\|\xi\widehat{g}_\xi\|_{L^2_\xi}\big)\Big)\,\dd s\nonumber\\
   &+C\int_0^{\rho}\big\||\xi|^{\mu}\tilde{K}_s(\rho,s;\xi)\big\|_{L^2_\xi}\|A-A_0\|_{L^{\infty}_v}\|\widehat{g}\|_{L^2_\xi}\,\dd s\nonumber \\
   \leq& \,C\int_0^\rho\rho^{1-\frac{\theta}{2}-\mu\theta}
     \Big(s^{-2+2\theta}s^{1+2\theta+\frac{3\theta}{2}}+\eps s^{-2+2\theta}s^{1+2\theta+\frac{3\theta}{2}} + s^{-1+\frac{3\theta}{2}}\Big)\,\dd s\nonumber\\
   \leq &\, C\rho^{1+\theta-\mu\theta}.
 \label{4.18a} \end{align}
Moreover, we observe that, for $0\leq\mu<\half$,
$$
\int_\R |\xi|^\mu|\xi\widehat{g}_\xi(\rho,\xi)|\,\dd \xi\leq C\rho^{1+\theta-\mu\theta}.
$$
To extend these inequalities to $\mu\geq\half$, we note that, for all $0\leq\mu<\nu + \frac{3}{2}$,
$$
\Big\||\xi|^{\mu-1}\frac{J_{\nu+1}(\xi k(s))}{(\xi k(s))^{\nu+1}}\Big\|_{L^1_\xi}\leq s^{1+2\theta-\mu\theta}.
$$
Distributing the powers of $\xi$ appropriately within the integral and estimating as \eqref{4.18a} in a straightforward way,
we obtain
$$
|\partial_v^\mu g(\rho,v)|\leq C\rho^{1+\theta-\mu\theta}
\qquad\mbox{for all $0<\mu<\la+2$}.
$$
Similarly, we have
$$
\rho|\partial_\rho\partial_v^\mu g(\rho,v)|\leq C\rho^{1+\theta-\mu\theta}
\qquad\mbox{for all $0<\mu<\la+2$}.
$$
Thus, by the standard embedding of the weighted Sobolev space $W^{1,p}_\rho\subset C^{0,\al}_\rho$,
we obtain that $\partial_v^\mu g(\rho,v)$ is H\"{o}lder continuous.

To conclude the proof,
we observe that $[1-z^2]_+^\mu$, for $z=\frac{v}{k(\rho)}$,
is positive on the support of $g(\rho,v)$, and
$$
|\partial_z^\mu g(\rho,v)|=k(\rho)^{\mu}|\partial_v^\mu g(\rho,v)|.
$$
Moreover, by the H\"older continuity above, we have
$$
|\partial_z^\be g|\leq \sup |\partial_z^\mu g|[1-z^2]_+^{\mu-\be}.
$$
We then calculate that, for $0<\be<\mu$,
\beqas
|\partial_v^\be g(\rho,v)|=&\,Ck(\rho)^{-\be}|\partial_z^\be g(\rho,v)|\\
\leq&\, C\rho^{-\be\th}\big(\sup_{z}|\partial_z^\mu g(\rho,v)|\big)[1-z^2]_+^{\mu-\be}\\
\leq& \,C\rho^{-\be\th+1+\th}[1-z^2]_+^{\mu-\be}\\
\leq& \,C\rho^{1+(1-2\mu+\beta)\th}[k(\rho)^2-v^2]_+^{\mu-\be}.
\eeqas
This completes the proof of Theorem \ref{thm:entropykernelregularity}, and hence Theorem \ref{thm:entropykernelmain}.
\end{proof}

\begin{thm}\label{thm:chipositive}
As $\eps\to0$, the relativistic entropy kernel $\chi(\rho,v)$ converges uniformly to the classical entropy kernel $\chi^*(\rho,v)$
on $\{\rho\leq\rho_M\}$.
In particular, in the $(k,v)$--coordinates,
$$
|\chi(k,v)-\chi^*(k,v)|\leq C\eps\chi^*(k,v),
$$
where $C$ depends only on $\rho_M$, which implies that, when $\eps$ is sufficiently small,
$$
\chi(k,v)\geq \half\, \chi^*(k,v)>0
$$
in the interior of its support $\{|v|\leq k\}$.
\end{thm}

\begin{proof}
It is clear that, as $\eps\rightarrow0$,
 $$
 k(\rho):=\int_{0}^{\rho}\frac{\sqrt{p'(s)}}{s + \eps p(s)}\,\dd s\longrightarrow k^*(\rho):=\int_0^\rho\frac{\sqrt{p'(s)}}{s}\,\dd s
 \qquad \mbox{uniformly in $\rho\in [0, \rho_M]$}.
 $$
 Also, from the expressions given in \S\ref{sec:entropykernelcoeffs}, we see that coefficients $a_i$, $i=1,2$, converge uniformly to
 their classical counterparts ($a_\sharp$ and $a_\flat$ in the notation of \cite{ChenLeFloch}).
 Thus, it suffices to show that the remainder function $g(\rho,v)$, determined in \S\ref{sec:entropykernelproofs},
 converges to the remainder function $g_*$ of the classical kernel.

 Recall that $g$ is defined as the fixed point of
 $$
 \widehat{g}(\rho,\xi)=\int_0^{\rho}\tilde{K}(\rho,s;\xi)\F(\mathcal{S}(g))(s,\xi)\,\dd s,
 $$
 where $\tilde{K}$ is defined in \eqref{eq:Ktilde}, $\mathcal{S}$ is defined in Proposition \ref{prop:Pgexpression},
 and $g^*$ is defined as the fixed point of
 $$
 \widehat{g^*}(\rho,\xi)=\int_0^{\rho}\tilde{K^*}(\rho,s;\xi)\F(\mathcal{S^*}(g^*))(s,\xi)\,\dd s,
 $$
 where $\tilde{K^*}$ and $\mathcal{S^*}$ can be found in \cite{ChenLeFloch2}.

From the expressions for $\tilde{K^*}$ and $\mathcal{S^*}$ in \cite{ChenLeFloch2}, we see that the difference, $\tilde{K}-\tilde{K^*}$, satisfies
$$
|\tilde{K}(\rho,s;\xi)-\tilde{K^*}(\rho,s;\xi)|
\leq C\eps\Big(\frac{k(\rho)}{k'(\rho)}\Big)^\half\Big(\frac{k(s)}{k'(s)}\Big)^\half Q_\nu(\xi k(\rho))Q_\nu(\xi k(s))^{-1}R(\xi k(s))
$$
for $0\leq s\leq\rho$ and $\xi\in\R$.
Similarly,
$$
|\F(\mathcal{S}(g)-\mathcal{S^*}(g^*))|\leq C\eps |\widehat{f}_{\la+1}(\xi k(\rho))|\rho^{-1+2\th}+\eps^2 |r(\rho,\xi)|.
$$
Using
 \begin{align*}
 \widehat{g}(\rho,\xi)-\widehat{g^*}(\rho,\xi)
 =&\int_0^{\rho}\tilde{K}(\rho,s;\xi)\F(\mathcal{S}(g))(s,\xi)\,\dd s
   -\int_0^{\rho}\tilde{K^*}(\rho,s;\xi)\F(\mathcal{S^*}(g^*))(s,\xi)\,\dd s\nonumber\\
   =&\int_0^{\rho}\big(\tilde{K}-\tilde{K^*}\big)(\rho,s;\xi)\F(\mathcal{S}(g))(s,\xi)\,\dd s\nonumber\\
   &+\int_0^{\rho}\tilde{K^*}(\rho,s;\xi)\F(\mathcal{S}(g)-\mathcal{S^*}(g^*))(s,\xi)\,\dd s,
\end{align*}
we may argue as in the proof of Theorem
\ref{thm:entropykernelregularity} to conclude the bound we want.
\end{proof}

\section{The Weak Entropy-Flux Kernel}\label{sec:entropyfluxkernel}
With Theorem \ref{thm:entropykernelmain}, thereby demonstrating the existence of the entropy kernel,
we now move on to the entropy-flux kernel and Theorem \ref{thm:entropyfluxkernelmain}.

To this end, we consider the expressions derived for general entropy pairs:
\beqas
q_{\rho} =\frac{u(1-\eps p')}{1-\eps^2 p' u^2}\eta_{\rho} + \frac{p'(1-\eps u^2)^2}{(1-\eps^2 p'u^2)(\rho + \eps p)}\eta_u,\quad
q_u =\frac{\rho + \eps p}{1 -\eps^2 p'u^2}\eta_{\rho} + \frac{u(1-\eps p')}{1 - \eps^2 p'u^2}\eta_u.
\eeqas
Changing the variables to $(\rho,v)$ and
setting $(\tilde{\rho}, \tilde{u}):=(\frac{(\rho + \eps p)(1-\eps u^2)}{1 -\eps^2 p'u^2}, \frac{u(1-\eps p')}{1-\eps^2 p' u^2})$ yield
\beqas
q_{\rho} = \,\tilde{u}\eta_{\rho} + \tilde{\rho}k'^2\eta_v,\quad q_v = \,\tilde{\rho}\eta_{\rho} + \tilde{u}\eta_v.
\eeqas

\begin{lemma}\label{lemma:utilderhotildeexpansions}
$\tilde{\rho}(\rho,v)$ and $\tilde{u}(\rho,v)$ can be expanded as
\begin{align}
&\tilde{\rho}(\rho,v)=\rho_0(\rho)+\eps\rho_1(\rho)(k(\rho)^2-v^2)+\eps^2\rho_2(\rho,v)(k(\rho)^2-v^2)^2,\label{eq:tilderhoexpansion}\\
&\tilde{u}(\rho,v)=v\big(u_0(\rho)+\eps u_1(\rho)(k(\rho)^2-v^2)+\eps^2u_2(\rho,v)(k(\rho)^2-v^2)^2\big),\label{eq:tildeuexpansion}
\end{align}
with
\begin{align*}
  \rho_0(\rho)=\frac{(\rho+\eps p)(1-\eps u(k)^2)}{1-\eps^2 p'u(k)^2},\qquad
  u_0(\rho)=\frac{u(k)(1-\eps p')}{k(1-\eps^2 p'u(k)^2)},
\end{align*}
 and
$$
|\rho_1(\rho)|+|\rho_2(\rho,v)|\leq C\big(\rho+\eps p(\rho)\big),\quad
|u_1(\rho)|+|u_2(\rho,v)|\leq C,
$$
where $C>0$ is independent of $\eps>0$.
\end{lemma}

The proof of this lemma is the same as that of Lemma \ref{lemma:A,Bdecomposition}, by taking a Taylor expansion in $v^2$ around $k(\rho)^2$.

We now need to derive an equation for the entropy-flux kernel.
We recall the operator:
$$
\mathbf{L}=\partial_{\rho\rho}-k'^2\partial_{vv} + \eps A\partial_\rho + \eps Bv\partial_v
$$
and observe that, by definition (see \eqref{eq:A,Bintro}), $\eps A=\frac{\tilde{\rho}_\rho-\tilde{u}_v}{\tilde{\rho}}$
and $\eps Bv=\frac{\tilde{u}_\rho-k'^2\tilde{\rho}_v}{\tilde{\rho}}$.
We may therefore calculate
\begin{align}
\mathbf{L}(\sigma-\tilde{u}\chi)=&\,\sigma_{\rho\rho}-k'^2\sigma_{vv} + \eps A\sigma_\rho + \eps Bv\sigma_v -2\tilde{u}_\rho\chi_\rho-\tilde{u}\chi_{\rho\rho}+2k'^2\tilde{u}_v\chi_v \nonumber\\
   &+ k'^2\tilde{u}\chi_{vv}- \eps A\tilde{u}\chi_\rho - \eps Bv\tilde{u}\chi_v-\mathbf{L}(\tilde{u})\chi\nonumber\\
=&-\tilde{u}_\rho\chi_\rho + 2k'k''\tilde{\rho}\chi_v + k'^2\tilde{\rho}_\rho\chi_v-k'^2\tilde{\rho}_v\chi_\rho + k'^2\tilde{u}_v\chi_v\nonumber\\
   &+ \eps Ak'^2\tilde{\rho}\chi_v + \eps Bv\tilde{\rho}\chi_\rho - \mathbf{L}(\tilde{u})\chi\nonumber\\
=&\,2k'k''\tilde{\rho}\chi_v+2k'^2\tilde{\rho}_\rho\chi_v-2k'^2\tilde{\rho}_v\chi_\rho-\mathbf{L}(\tilde{u})\chi\nonumber\\
=&\,\tilde{F}(\rho,v), \label{eq:fluxequationrighthandside}
\end{align}
as claimed in \eqref{eq:modifiedentropyfluxequation}.
Defining $F(\rho,v):=\tilde{F}(\rho,v)+\mathbf{L}(\tilde{u}\chi)$, we derive \eqref{eq:entropyfluxequation}.

\subsection{The coefficients for the entropy-flux kernel}\label{sec:entropyfluxkernelcoeffs}
As  with the entropy kernel in \S \ref{sec:entropykernelcoeffs},
we now derive the expressions for the coefficients of the entropy-flux kernel, $b_1(\rho)$ and $b_2(\rho)$.

Expanding $\tilde{F}(\rho,v)$ in the coefficients of $G_{\la}(\rho,v)$ (recall from \S \ref{sec:entropykernelcoeffs}
that $G_\la(\rho,v)=[k(\rho)^2-v^2]_+^\la$), we find
\begin{align}
\mathbf{L}(\sigma-\tilde{u}\chi)(\rho,v)
=&\,-4\la vk'(k'\rho_0)'a_1G_{\la-1}(\rho,v)\nonumber\\
&\,-4vG_{\la}(\rho,v)\big((\la+1)k'(k'\rho_0)'a_2+\la\eps k'(k'\rho_1)'a_1-\eps k'^2\rho_1a_1'+\frac{1}{4}\om a_1\big)\nonumber\\
&\,+v\tilde{f}(\rho,v),\label{eq:firstFexpansion}
\end{align}
where
\beqs
\om(\rho) =u_0''+\eps A_0u_0'+\eps B_0u_0 +4\eps kk'u_1'
+2\eps\big(4 k'^2 + kk''+\eps A_0kk'-\eps B_0k^2\big)u_1,
\eeqs
and $\tilde{f}(\rho,v)$ is a more regular term satisfying the bound:
$$
|\tilde{f}(\rho,v)|\leq C\rho^{2\th-2}G_{\la+1}.
$$
Moreover, we note from the asymptotics for $k(\rho)$ in Lemma \ref{lemma:k-asymptotics} and
the bounds for $a_1(\rho)$ and $a_2(\rho)$ in Theorem \ref{thm:entropykernelmain} that
the coefficients of $G_{\la-1}(\rho,v)$ and $G_\la(\rho,v)$ also satisfy the bounds of form $C\rho^{2\th-2}$.

We are now looking for the entropy-flux kernel in the form:
$$
(\sigma-\tilde{u}\chi)(\rho,v)=-v\big(b_1(\rho)G_\la(\rho,v)+b_2(\rho)G_{\la+1}(\rho,v)\big)+h(\rho,v).
$$
We write $b_i(\rho)=W_i(\rho)a_i(\rho)$, $i=1,2$, below to emphasize the relation between coefficients $b_i(\rho)$ and $a_i(\rho)$.
Applying operator $\mathbf{L}$ to this ansatz, we have
\begin{align}
\mathbf{L}(\sigma-\tilde{u}\chi)
=&\,-4\lambda v a_1k'(kW_1)' G_{\la-1}(\rho,v)\nonumber\\
&\,+G_\la(\rho,v)v\big(-W_1''a_1-2W_1'a_1'-4(\lambda+1)kk'a_2(W_2'+\frac{k'}{k}W_2)\nonumber\\
&\,\qquad\qquad\qquad +(W_2-W_1)a_1''-\eps A_0W_1'a_1-\eps B_0W_1a_1+\eps A_0(W_2-W_1)a_1'\nonumber\\
&\,\qquad\qquad\qquad +2\lambda \eps B_0(W_2-W_1)a_1-2\lambda \eps^2 kk'A_1a_1(W_1-W_2)\nonumber\\
&\,\qquad\qquad\qquad +2\lambda\eps^2 (W_1-W_2)B_1a_1k^2\big)\nonumber\\
&\quad +G_{\la+1}(\rho,v)v\big(-W_2''a_2-2W_2'a_2'-W_2a_2''-2\lambda \eps^3 A_2W_1a_1kk'\nonumber\\
&\qquad\qquad\qquad\qquad  +2\lambda \eps^3 B_2v^2W_1a_1-\eps (A-A_0)W_1'a_1-\eps (A-A_0)W_1a_1'\nonumber\\
&\qquad\qquad\qquad\qquad -2(\lambda+1)\eps (A-A_0)W_2a_2kk'-2\lambda\eps (B-B_0)W_1a_1\nonumber\\
&\qquad\qquad\qquad\qquad  +2(\lambda+1)\eps (B-B_0)v^2W_2a_2-\eps AW_2'a_2-\eps AW_2a_2'\nonumber\\
&\qquad\qquad\qquad\qquad  -\eps BW_2a_2-2(\lambda+1)\eps B_0W_2a_2\big)\nonumber\\
&\quad +\mathbf{L}(h). \label{eq:secondFexpansion}
\end{align}
Matching the coefficients in the terms of $G_{\la-1}(\rho,v)$
in \eqref{eq:firstFexpansion} and \eqref{eq:secondFexpansion}, $W_1(\rho)$ must solve
$$
(kW_1)'(\rho)=(k'\rho_0)'(\rho).
$$
This has the solution
\beq\label{eq:V_1coeff}
W_1(\rho)=\frac{k'(\rho)}{k(\rho)}\rho_0(\rho).
\eeq
For coefficient $W_2$, comparing the next most singular terms yields
\beqas
W_2'(\rho)+\frac{k'(\rho)}{k(\rho)}W_2(\rho)-\frac{\tilde{W}(\rho)}{4(\lambda+1)k(\rho)k'(\rho)a_2(\rho)}W_2(\rho)
=\frac{\tilde{\Omega}(\rho)}{4(\lambda+1)k(\rho)k'(\rho)a_2(\rho)},
\eeqas
where $\tilde{W}(\rho)$ is defined as in \eqref{eq:W}, and $\tilde{\Om}$ is given by
\begin{align}
  \tilde{\Om}=&\,-W_1''a_1-2W_1'a_1'-W_1\tilde{W}-\eps A_0W_1'a_1 -\eps B_0W_1a_1\nonumber\\
   &\,+4(\la+1)k'(k'\rho_0)'a_2+4\la\eps k'(k'\rho_1)'a_1-4\eps k'^2\rho_1a_1'+\om a_1.\label{eq:Omegatilde}
 \end{align}
Using the expressions for $(W_1,\tilde{u},\tilde{\rho})$ in \eqref{eq:V_1coeff} and Lemma \ref{lemma:utilderhotildeexpansions},
and applying once again Lemma \ref{lemma:k-asymptotics}, we see that $\tilde{\Omega}(\rho)=O\left(\rho^{2\th-2}\right)$.

From \eqref{eq:a_2coefficient} and the expression for $a_2(\rho)$,  we have
$$
-\frac{\tilde{W}(\rho)}{4(\lambda+1)k(\rho)k'(\rho)a_2(\rho)}
=\frac{\dd}{\dd\rho}\Big(\log\int_0^\rho e^{-\tilde{a}(\tau)}k(\tau)^\la k'(\tau)^{-\half}\tilde{W}(\tau)\,\dd\tau\Big).
$$
Now we define an integrating factor:
\beqas
I(\rho)=&\,\exp\Big\{\int^\rho\Big( \frac{k'(s)}{k(s)}+\frac{\dd }{\dd s}\Big(\log\int_0^s e^{-\tilde{a}(\tau)}k(\tau)^\la k'(\tau)^{-\half}\tilde{W}(\tau)\,\dd \tau\Big)\Big)\,\dd s\Big\}\\
 =&\, k(\rho)\int_0^\rho e^{-\tilde{a}(\tau)}k(\tau)^\la k'(\tau)^{-\half}\tilde{W}(\tau)\,\dd \tau.
\eeqas
One may check $I(\rho)=O(\rho^\theta)$ as $\rho\to 0$.
Thus, $\frac{\tilde{\Omega}(\rho)}{4(\lambda+1)k(\rho)k'(\rho)a_2(\rho)}I(\rho)=O(\rho^{-1+\theta})$ as $\rho\to0$, and hence is integrable.
Therefore, we obtain
\beq\label{eq:V_2coeff}
W_2(\rho)=I(\rho)^{-1}\int_0^{\rho}\frac{\tilde{\Omega}}{4(\lambda+1)kk'a_2}(s)I(s)\,\dd s.
\eeq
It is simple to verify that $|W_2(\rho)|\leq C$, where $C$ depends only on $\rho_M$.

\subsection{Proof of the existence and regularity of the entropy-flux kernel}\label{sec:entropyfluxkernelproofs}
As for the remainder function in the expansion of the entropy kernel,
we derive an equation for $\tilde{\mathbf{L}}h$,
where $\tilde{\mathbf{L}}$ is defined as in \eqref{def:P}.

\begin{prop}
The remainder function $h(\rho,v)$ satisfies
 \begin{align*}
  \F(\tilde{\mathbf{L}}h)(\rho,\xi)
  =&\,\mathcal{F}({\mathcal{T}(g)})(\rho,\xi)\\
  :=&-\mathcal{F}(\eps(A-A_0)g_\rho)-\mathcal{F}(\eps Bvg_v)-\beta(\rho)\hat{g}+H_1(\rho)k(\rho)\xi\widehat{f}_{\lambda+2}(k(\rho)\xi)
    +\eps^2 r(\rho,\xi),
\end{align*}
where $H_1(\rho)=O(\rho^{-1+2\theta})$ as $\rho\rightarrow0$
and $r(\rho,\xi)=O(\rho^{-1+2\th}(k(\rho)|\xi|)^{-\la-1-\al-\half})$ as $\xi\to\infty$, for some $\al>0$.
\end{prop}
In particular, $r(\rho,\xi)$ acts asymptotically
like $\widehat{f}_{\la+1+\al}(k(\rho)\xi)$ as $\xi\to\infty$.
As the proof is very similar to that of Proposition \ref{prop:Pgexpression}, we omit it.

Observe that
\beqs
k(\rho)|\xi\widehat{f}_{\la+2}(k(\rho)\xi)|\leq\begin{cases}
C & \text{ if } k(\rho)|\xi|\leq 1,\\[1mm]
C|k(\rho)\xi|^{-\la-\frac{3}{2}} & \text{ if } k(\rho)|\xi|> 1,
\end{cases}
\eeqs
so that $k(\rho)\xi\widehat{f}_{\la+2}(k(\rho)\xi)$ satisfies the same bound
as $\widehat{f}_{\la+1}(k(\rho)\xi)$.
We may prove the following theorem analogously to Theorems \ref{thm:entkernelexistence}
and \ref{thm:entropykernelregularity}--\ref{thm:chipositive}.

\begin{thm}[Existence and Regularity of the Entropy-Flux Kernel]\label{th:fluxremainderexistenceandregularity}
For $\gamma \in (1,3)$, there exists $\widehat{h} \in L^{\infty}(0,\rho_{M};L^2(\mathbb{R}))$ that is a fixed point of
 \beq\label{eq:hfixedpoint}
  \widehat{h}(\rho,\xi)=\int_0^\rho\tilde{K}(\rho,s;\xi)\F(\mathcal{T}(h))(s,\xi)\,\dd s
  \eeq
such that $\widehat{h}(\rho,\xi)$ satisfies
\beq\label{ineq:hhatL2bound}
  \|\widehat{h}(\rho,\cdot)\|_{L^2}+\|\rho\widehat{h}_\rho(\rho,\cdot)\|_{L^2}\leq C\rho^{1+\frac{3\theta}{2}}e^{\frac{\rho^{2\theta}}{2\theta}}.
  \eeq
The remainder function $h=h(\rho,v)$ is such that $\partial_v^\mu h$ is H\"{o}lder continuous in $(\rho,v)$ for $\rho>0$
for all $\mu$ with $0\leq\mu<\lambda+2$.
In addition, if $0<\be<\mu$,
$$
|\partial_v^\be h(\rho,v)|\leq C\rho^{1+\th-2\mu\th+\be\th}[k(\rho)^2-v^2]_+^{\mu-\be}.
$$
Finally, as $\eps\to0$, $\sigma(\rho,v,s)\to\sigma^*(\rho,v,s)$ locally uniformly,
where $\sigma^*$ is the classical entropy-flux kernel as in {\rm \cite{ChenLeFloch,ChenLeFloch2}}.
In particular,
$$
\|\sigma(\rho,\cdot,\cdot)-\sigma^*(\rho,\cdot,\cdot)\|_{L^\infty(\{0\leq\rho\leq\rho_M\})}\leq C\eps.
$$
\end{thm}
In what follows, especially in \S \ref{sec:relativisticcompactness},
we require not only an expansion for $\sigma-\tilde{u}\chi$,
but also for $\sigma-\la_\pm\chi$, where $\la_\pm$ are the eigenvalues of the system defined in \S 2.

\begin{cor}\label{cor:fluxexpansioneigenvalues}
$\sigma-\la_\pm\chi$ satisfy the following expansions{\rm :}
\begin{equation*}
   (\sigma-\lambda_{\pm}\chi)(\rho,v,s)=(\sigma-\lambda_{\pm}\chi)(\rho,v-s,0)=(\mp k-(v-s))\rho_0\frac{k'}{k}\chi(\rho,v-s)+R_\pm(\rho,v-s),
 \end{equation*}
where $|R_\pm(\rho,v-s)|\leq C|k(\rho)^2-(v-s)^2|\chi(\rho,v-s)$.
\end{cor}

\begin{proof}
We begin by recalling
$$
\lambda_{\pm}=\tilde{u}\pm \tilde{\rho}k'(\rho).
$$
To show that $\sigma-\la_\pm\chi$ remain invariant under the Lorentzian transformation,
it suffices to check that the functions and their derivatives with respect to $\rho$ at $\rho=0$ remain invariant.
However, this follows from the simple fact that
$$
\big(\tilde{\rho}k'(\rho)\chi(\rho,v-s)\big)_\rho\to 0
\qquad\mbox{in the sense of measures as $\rho\to 0$}.
$$
Considering the case for $s=0$ without loss of generality,
we use expansions \eqref{eq:relativisticentropyexpansion} and \eqref{eq:relativisticentropyfluxexpansion}
for $\chi(\rho,v)$ and $(\sigma-\tilde u\chi)(\rho,v)$ to obtain
\beqas
  (\sigma-\lambda_{\pm}\chi)(\rho,v,0)=&\,(\mp k-v)\rho_0\frac{k'}{k}a_1[k^2-v^2]_+^\lambda\\
  &-\big(\pm \rho_1a_1k'\pm \rho_2(k^2-v^2)a_1k'\pm \tilde{\rho}a_2k'+vb_2\big)[k^2-v^2]_+^{\lambda+1}\\
  &\mp \tilde{\rho}k'g(\rho,v)+h(\rho,v),
\eeqas
from which the desired conclusion follows directly.
\end{proof}

In the reduction argument of the next section, we require an accurate analysis of the singularities
of the entropy and entropy-flux kernels constructed above.
To this end, we now provide explicit formulae for the singularities in the fractional
derivatives of order $\la+1$.

\begin{prop}[Explicit Singularities of the Entropy Kernels]\label{prop:relativisticexplicitsingularities}
The two distributions $\partial_v^{\lambda+1}\chi$ and $\partial_v^{\lambda+1}\sigma$ satisfy
\beqa
&\partial_v^{\lambda+1}\chi=\,k'(\rho)^{-\frac{1}{2}}e^{\tilde{a}(\rho)}\sum_\pm K^\pm\delta_{v=\mp k(\rho)} + e^I(\rho,v),\\
&\partial_v^{\lambda+1}(\sigma-\tilde{u}\chi)
=-v\rho_0(\rho)k(\rho)^{-1}k'(\rho)^{\frac{1}{2}}e^{\tilde{a}(\rho)}\sum_\pm K^\pm\delta_{v=\mp k(\rho)} + e^{II}(\rho,v),
\eeqa
where $K^\pm$ are constants, and $e^I$ and $e^{II}$ are H\"{o}lder continuous functions in the interior of the support of the kernels
such that
\begin{align}
  &|e^I(\rho,v)|\leq\, Ck(\rho)^{\lambda-1+2\alpha}[k(\rho)^2-v^2]_+^{-\alpha},\label{ineq:eIbound}\\
  &|e^{II}(\rho,v)|\leq\, Ck(\rho)^{\lambda+2\alpha}[k(\rho)^2-v^2]_+^{-\alpha}\label{ineq:eIIbound}
\end{align}
for all $\alpha\in(0,1]$.
\end{prop}

\begin{proof}
The identities for the fractional derivatives $\partial_v^{\la+1}G_{\la}$ and $\partial_v^{\la}G_\la$
may be found in \cite[Proof of Proposition 2.4]{ChenLeFloch}.
The desired representations then follow from expansions \eqref{eq:relativisticentropyexpansion} and \eqref{eq:relativisticentropyfluxexpansion},
exactly as in that proof.
\end{proof}

Finally, we record a property of the coefficients to be required in the sequel.

\begin{prop}\label{prop:Dcoefficient}
We define a coefficient $D=D(\rho)$ as
$$
D(\rho):=a_1(\rho)b_1(\rho)-2k(\rho)^2\big(a_1(\rho)b_2(\rho)-a_2(\rho)b_1(\rho)\big).
$$
Then there exists $\eps_0>0$ such that, for all $\eps\in(0,\eps_0)$, $D(\rho)>0$ for any $\rho\in(0,\rho_M)$.
\end{prop}

\begin{proof}
This follows from the formulae and bounds given for the coefficients above
in Theorems \ref{thm:entropykernelmain} and \ref{thm:entropyfluxkernelmain}.
Indeed, a calculation shows
$$
D(\rho)=\frac{a_1(\rho)^2}{2(\la+1)p'(\rho)}\big(\rho p''(\rho)+2p'(\rho)+O(\eps\rho^{\ga-1})\big).
$$
\end{proof}

\section{Compactness Framework}\label{sec:relativisticcompactness}
Now that the entropy and entropy-flux kernels have been constructed,
we apply them to analyze the compactness properties of a sequence of solutions or approximate
solutions of system \eqref{eq:relativistic-Euler}.
Therefore, the principal aim of this section is to show that a uniformly bounded sequence
of functions satisfying the $H^{-1}_{\loc}$ compactness of the entropy dissipation
measures (see \eqref{ass:relativisticentropydissipation} below) can be shown to converge not only weakly,
but also strongly in $L^p_{\loc}$.

\begin{thm}[Compactness of Approximate Solutions]\label{thm:relativisticcompactness}
Let $(\rho^\de, u^\de)\in (L^\infty(\R^2_+))^2$ with $\rho^\de\geq 0$ be a function sequence for $\de\in(0,1)$ such that
\begin{equation*}
 |u^\de(t,x)|\leq M<\frac{1}{\sqrt{\eps}},  \quad  0\leq\rho^\de(t,x)\leq\rho_{M}<\rho_{\max}^\eps \qquad \text{ for a.e. $(t,x)\in\R^2_+$}
\end{equation*}
for $\eps\in (0, \eps_0]$ with some $\eps_0>0$ defined as in Proposition {\rm \ref{prop:Dcoefficient}},
and $M$ and $\rho_M$ independent of $\de>0$.
Suppose that
the sequence of entropy dissipation measures
\begin{equation}\label{ass:relativisticentropydissipation}
\eta(\rho^\de,u^\de)_t + q(\rho^\de,u^\de)_x \qquad \text{ is compact in $H^{-1}_{\rm loc}(\R^2_+)$}
\end{equation}
for any weak entropy pair $(\eta,q)$.
Then there exist a subsequence $($still denoted$)$ $(\rho^\de,u^\de)$ and
measurable functions $(\rho,u)$ such that
\begin{equation*}
|u(t,x)|\leq M,  \quad  0\leq\rho(t,x)\leq \rho_M  \qquad \text{ for a.e. $(t,x)\in\R^2_+$},
 \end{equation*}
and $(\rho^\de,u^\de)$ converges strongly to $(\rho,u)$ as $\delta\to 0$ in $L^r_{\rm loc}(\R^2_+)$ for
all $r\in[1,\infty)$.
\end{thm}

The proof of this theorem rests on two main ingredients: the div-curl lemma of Murat-Tartar \cite{Murat1,Tartar},
and the following reduction result for Young measures constrained by the Tartar commutation relation
whose proof is temporarily postponed.

\begin{thm}[Reduction of Support of the Young Measure]\label{thm:relativisticreduction}
Suppose that $\nu(\rho,v)$ is a Young measure $($probability measure$)$ with bounded support contained
in $\{|v|\le v(M), 0\leq\rho\leq\rho_{M}\}$
satisfying the commutation relation{\rm :}
\begin{equation}\label{eq:relativisticcommutation}
 \overline{\eta_1q_2-\eta_2q_1}=\overline{\eta_1}\,\overline{q_2}-\overline{\eta_2}\,\overline{q_1}
\end{equation}
for any two weak entropy pairs $(\eta_1,q_1)$ and $(\eta_2,q_2)$,
where we have denoted
$$
\overline{f}:=\int f(\rho,v)\,\dd \nu(\rho,v) \qquad\text{ for any continuous function $f(\rho,v)$.}
$$
Then either $\nu$ is supported in the vacuum line $\{\rho=0\}$ or the support of $\nu$ is a single point.
\end{thm}

\begin{proof}[Proof of Theorem {\rm \ref{thm:relativisticcompactness}}]
For convenience, we work in the $(\rho,v)$--coordinates.
As the function sequence $(\rho^\de, v^\de)$ with $v^\de=v(u^\de)$
is uniformly bounded in $L^\infty(\R^2_+)$, we may
extract a weakly-star convergent subsequence (still labelled as) $(\rho^\de, v^\de)\weakstarto (\rho,v)$
in $L^\infty(\R^2_+)$.
By the fundamental theorem
of Young measures (\textit{cf}.~\cite{Ball}),
we may associate, for {\it a.e.} $(t,x)\in\R^2_+$,
a probability
measure $\nu_{t,x}$ such that, for {\it a.e.} $(t,x)$, $\supp\,\nu_{t,x}\subset\{|v|\le v(M), 0\leq \rho\leq \rho_{M}\}$ and
$$
f(\rho^\de(t,x),v^\de(t,x))\,\stackrel{*}{\rightharpoonup} \,\overline{f}(t,x)= \int_{\mathbb{R}^2_+}f(\rho,v)\,\dd \nu_{t,x}(\rho, v)
\qquad \text{ in $L^\infty(\R^2_+)$}
$$
for any continuous function $f:\mathbb{R}^2_+\rightarrow\mathbb{R}$.

We take any two weak entropy pairs $(\eta_1,q_1)$ and $(\eta_2,q_2)$ and, for simplicity of notation,
we define $\eta_i^\de=\eta_i(\rho^\de(t,x),v^\de(t,x))$ and $q_i^\de=q_i(\rho^\de(t,x),v^\de(t,x))$ for $i=1,2$.
Then, by definition of the Young measure, we have
$$
\eta_1^\de\, q_2^\de-\eta_2^\de\, q_1^\de\, \stackrel{*}{\rightharpoonup}\, \overline{\eta_1q_2-\eta_2q_1}\qquad
\text{ in $L^\infty(\R^2_+)$ as $\de\rightarrow0$.}
$$
On the other hand, by the $H^{-1}$-compactness assumption \eqref{ass:relativisticentropydissipation},
we may apply the div-curl lemma ({\it cf.}  \cite{Murat1,Tartar}) to sequences $w_1^\de=(\eta_1^\de,q_1^\de)$ and
$w_2^\de=(q^\de_2,-\eta^\de_2)$ to obtain
$$
\eta^\de_1\,q^\de_2-\eta^\de_2\,q^\de_1\, \rightharpoonup \, \overline{\eta_1}\,\overline{q_2}-\overline{\eta_2}\,\overline{q_1}
\qquad\,\,\,\,\mbox{in the sense of distributions on $\R^2_+$ as $\de\rightarrow0$}.
$$
Thus, by uniqueness of weak limits, for {\it a.e.} $(t,x)$,
we have the Tartar commutation relation:
\begin{equation*}
\overline{\eta_1q_2-\eta_2q_1}
=\overline{\eta_1}\,\overline{q_2}-\overline{\eta_2}\,\overline{q_1}
\end{equation*}
for any two weak entropy pairs $(\eta_1,q_1)$ and $(\eta_2,q_2)$.

We conclude by applying Theorem \ref{thm:relativisticreduction} to show that,
for {\it a.e.} $(t,x)$, $\nu_{t,x}$ is either constrained to a point, so that $\nu_{t,x}=\de_{(\rho(t,x),v(t,x))}$,
or $\nu_{t,x}$ is supported in the vacuum line: $\supp\,\nu_{t,x}\subset\{\rho=0\}$.
In either case, changing back to the conserved variables $U$ implies that the Young measure $\nu_{t,x}$
is a point mass {\it a.e.}.
Then we
conclude the strong convergence as claimed.
\end{proof}

The rest of this section is devoted to the proof of Theorem \ref{thm:relativisticreduction}.
As a preliminary step, we extend a result of DiPerna \cite{DiPerna2} for the classical Euler equations
to the relativistic case.
This lemma tells us that, in the $(w,z)$--plane, the smallest triangle containing the support of the Young measure $\nu_{t,x}$
(considered as a measure in $(w,z)$) must have its vertex in the support of $\nu_{t,x}$.
Note that the vacuum line $\{\rho=0\}$ corresponds to line $\{w=z\}$.

\begin{lemma}\label{lemma:relativisticvertexsupport}
Let $\nu$ be a probability measure on set $\{w\geq z\}$ with non-trivial support away from the vacuum line,
i.e.,~${\normalfont \supp}\,\nu\cap\{w>z\}\neq\emptyset$, and let $\nu$ further satisfy
the commutation relation \eqref{eq:relativisticcommutation}
for all weak entropy pairs $(\eta_1,q_1)$ and $(\eta_2,q_2)$. Let
 $$
 \{(w,z)\,:\,z_{\min}\leq z\leq w\leq w_{\max}\}
 $$
 be the smallest triangle containing the support of $\nu$ in the $(w,z)$-plane. Then its vertex
 $(w_{\max},z_{\min})$ belongs to ${\normalfont \supp}\,\nu$.
\end{lemma}

\begin{proof}
We argue by contradiction. Suppose that there exists $\alpha>0$ such that
$$
\supp\,\nu\cap\big([w_{\max}-\alpha,w_{\max}]\times[z_{\min},z_{\min}+\alpha]\big)=\emptyset.
$$
From the commutation relation \eqref{eq:relativisticcommutation} (dropping the test functions and working directly with the kernels),
we have
\begin{equation}\label{commut}
\frac{\overline{\chi(s_1)\sigma(s_2)-\chi(s_2)\sigma(s_1)}}{\overline{\chi(s_1)}\,\,\overline{\chi(s_2)}}
=\frac{\overline{\sigma(s_2)}}{\overline{\chi(s_2)}}-\frac{\overline{\sigma(s_1)}}{\overline{\chi(s_1)}}
\qquad\text{for $s_1,s_2\in\R$ such that $\overline{\chi(s_1)}\,\overline{\chi(s_2)}\ne 0$},
\end{equation}
where $\overline{\chi(s_j)}=\int\chi(\rho,v,s_j)\,\dd\nu(\rho, v), j=1,2$.
Setting $s_-:=z_{\min}$ and $s_+:=w_{\max}$,
we consider $s_1$ and $s_2$ such that $0<s_+-s_2<\alpha$ and $0<s_1-s_-<\alpha$.
As
$\supp \,\chi(s)=\supp\,\sigma(s)=\{z\leq s\leq w\}$,
we see that $(w,z)\in\supp(\chi(s_1)\sigma(s_2))$ implies that $(w,z)\not\in\supp\,\nu$.
Arguing in the same way for $\chi(s_2)\sigma(s_1)$,
we see that the left-hand side of \eqref{commut} vanishes.

We recall from Corollary \ref{cor:fluxexpansioneigenvalues} that
\beqs
 (\sigma-\lambda_{\pm}\chi)(\rho,v-s)=\big(\mp k-(v-s)\big)\rho_0\frac{k'}{k}\chi(\rho,v-s)+R_\pm(\rho,v-s),
\eeqs
where $|R_\pm(\rho,v-s)|\leq C|k(\rho)^2-(v-s)^2|\chi(\rho,v-s)$. Then
\begin{equation}\label{eq:eigenvaluedecomp}
\frac{\overline{\sigma(s)}}{\overline{\chi(s)}}
=\frac{\overline{\lambda_{\pm}\chi(s)}}{\overline{\chi(s)}}
+\frac{\overline{(\mp k-(v-s))\rho_0\frac{k'}{k}\chi(s)}}{\overline{\chi(s)}}+\frac{\overline{ R_\pm(s)}}{\overline{\chi(s)}}.
\end{equation}
We define the probability trace measures $\mu_+$ and $\mu_-$ by
\begin{align*}
&\frac{\overline{f\chi(s_2)}}{\overline{\chi(s_2)}}\rightarrow\langle\mu_+,f(w_{\max},\cdot)\rangle:=\int f(w_{\max},z)\,\dd \mu_+(z)
\qquad \text{ as }s_2\rightarrow s_+,\\[1mm]
&\frac{\overline{f\chi(s_1)}}{\overline{\chi(s_1)}}\rightarrow\langle\mu_-,f(\cdot,z_{\min})\rangle
:=\int f(w,z_{\min})\,\dd \mu_-(w) \qquad\,\, \text{ as }s_1\rightarrow s_-,
\end{align*}
for any continuous function $f=f(w,z)$.
It is now standard to see that these measures are well defined (\textit{cf}.~\cite{DiPerna2}).
We note that
\begin{align*}
&\left|\frac{\overline{(-k-(v-s_2))\rho_0\frac{k'}{k}\chi(s_2)}}{\overline{\chi(s_2)}}\right|
\leq C\max_{(w,z)\in\supp\,\nu\,\cap\{w\geq s_2\}}|w-s_2|\rightarrow 0 \qquad \text{ as }s_2\rightarrow s_+,\\
&\left|\frac{\overline{(k-(v-s_1))\rho_0\frac{k'}{k}\chi(s_1)}}{\overline{\chi(s_1)}}\right|
\leq C\max_{(w,z)\in\supp\,\nu\,\cap\{z\leq s_1\}}|z-s_1|\rightarrow 0 \qquad\quad \text{ as }s_1\rightarrow s_-.
\end{align*}
Moreover,
$$
\left|\frac{\overline{R_\pm(s)}}{\overline{\chi(s)}}\right|
\leq C\max_{(w,z)\in\,\supp\,\nu}[k^2-(v-s)^2]_+\leq C\max_{\supp\,\nu}|w-s_2||z-s_1|\rightarrow 0
$$
as either $s\to s_+$ or $s\to s_-$.
Thus, we deduce from \eqref{eq:eigenvaluedecomp} that
\begin{equation}\label{eq:turtle}
  \langle\mu_+,\lambda_+\rangle-\langle\mu_-,\lambda_-\rangle=0.
\end{equation}
Now suppose that $(w_{\max},z_1)$ and $(w_2,z_{\min})$ are the points on the edges of the triangle
with the coordinates given by $z_1=v_1-k(\rho_1)$ and
$w_2=v_2+k(\rho_2)$.
Observe that necessarily $v_1\geq v_2$ with equality only at the vertex of the triangle.
Now we calculate
\beqas
&\la_+(w_{\max},z_1)-\la_-(w_2,z_{\min})\\
&=\,\frac{(u_1-u_2)\big(1+\eps\sqrt{p'(\rho_1)}\sqrt{p'(\rho_2)}\big)
+\big(\sqrt{p'(\rho_1)}+\sqrt{p'(\rho_2)}\big)(1-\eps u_1u_2)}{\big(1-\eps u_1\sqrt{p'(\rho_1)}\big)\big(1-\eps u_2\sqrt{p'(\rho_2)}\big)}
>0,
\eeqas
as either $u_1>u_2$ or $u_1=u_2$ and $\rho_1=\rho_2>0$.
As both $\mu_+$ and $\mu_-$ are probability measures,
this gives the desired contradiction to \eqref{eq:turtle}.
\end{proof}

To prove Theorem \ref{thm:relativisticreduction}, we exploit the existence of an \it imbalance of regularity \rm in
the commutation relation \eqref{eq:relativisticcommutation}, following the approach
as developed in \cite{ChenLeFloch,LionsPerthameSouganidis} (also see \cite{DiPerna2}).

We write $P_j:=\partial_{s_j}^{\lambda+1}$, $j=2,3$, for the fractional derivative operators, and
define $\chi_j=\chi(\rho,v-s_j)$, $j=2,3$, and similarly for the other terms.
Then distributions $\overline{P_j\chi(s_j)}, j=2,3$, are defined as acting on test functions $\psi\in C_{\rm c}^\infty(\R)$ by
$$
\langle \overline{P_j\chi(s_j)},\psi\rangle
=-\int_\R\overline{\partial_{s_j}^\la\chi(s_j)}\psi'(s_j)\,\dd s_j \qquad \mbox{for $j=2,3$.}
$$
We choose standard (but distinct) mollifiers
$\phi_j\in C_{\rm c}^\infty(\R)$, $j=2,3$,
so that $\phi_j(s_j)\geq0$, $\int_\mathbb{R}\phi_j(s_j)\,\dd s_j=1$, and $\supp\,\phi_j(s_j)\subset (-1,1)$,
and set $\phi_j^\dee(s_j)=\frac{1}{\dee}\phi_j(\frac{s_j}{\dee})$ for $\dee>0$.

The strategy of the proof is first to apply operators $P_2$ and $P_3$ to the commutation
relation \eqref{eq:relativisticcommutation} and then to mollify them.
To make clear the claimed imbalance of regularity,
we make use of the fact that the limit of a mollified product of a measure with a BV function depends on the choice
of mollifiers used (\textit{cf}.~\cite{DMLFM}).
Mollifying the entropy and entropy-flux kernels and taking $s_2,s_3\to s_1$,
we obtain the expressions of form:
\beq\label{eq:mollificationexpression}
\overline{P_j\chi^\de_j}=\overline{P_j\chi_j}*\phi_j^\dee(s_1)
=\int\overline{\partial_{s_j}^\la\chi(s_j)}\frac{1}{\dee^{2}}\phi_j'(\frac{s_1-s_j}{\dee})\,\dd s_j
\qquad\mbox{for $j=2,3$}.
\eeq
Once we have differentiated and mollified the commutation relation \eqref{eq:relativisticcommutation},
we pass $\dee\to0$, relying on the properties of cancellation of singularities of the entropy and entropy-flux kernels
to obtain a limit depending on $\phi_2$ and $\phi_3$.
These properties are stated in the following lemma.

\begin{lemma}[Cancellation of singularities]\label{lemma:cancellationI}
As $\dee\to0$, we have the following convergence properties{\rm :}
\begin{enumerate}\vspace{-2mm}
\item[{\rm (i)}]  For $j=2,3$, functions $\chi_1P_j\sigma^\dee_j-\sigma_1P_j\chi_j^\dee$ are H\"{o}lder continuous in $(\rho,v,s_1)$
 and uniformly in $\dee$.
 Moreover, there exists a continuous function $X_1=X(\rho,v,s_1)$, independent of the choice of mollifying sequence,
 such that $\chi_1P_j\sigma^\dee_j-\sigma_1P_j\chi_j^\dee\rightarrow X_1$
 uniformly in $(\rho,v,s_1)$ when $\dee\rightarrow0$.

\smallskip
\item[{\rm (ii)}] $P_2\chi_2^\dee \,P_3\sigma_3^\dee-P_3\chi_3^\dee \,P_2\sigma_2^\dee$ are uniformly bounded
 measures such that, as $\dee\rightarrow0$,
$$
 P_2\chi_2^\dee\, P_3\sigma_3^\dee-P_3\chi_3^\dee\, P_2\sigma_2^\dee
 \,\rightharpoonup\,Y(\phi_2,\phi_3)M(\rho)D(\rho)\sum_\pm (K^\pm)^2\delta_{s_1=v\pm k(\rho)}
$$
weakly-star in measures in $s_1$ and uniformly in $(\rho,v)$,
where
$$
 Y(\phi_2,\phi_3)=\int_{-1}^1\int_{s}^{1}\big(\phi_3(t-s-1)\phi_2(t)-\phi_2(t-s-1)\phi_3(t)\big)\,\dd t\,\dd s,
$$
$M(\rho)=(\la+1)c_{*,\la}^{-2}k(\rho)^{2\la}$ for $\rho>0$, and $D(\rho)$ is as in Proposition {\rm \ref{prop:Dcoefficient}}.
\end{enumerate}
\end{lemma}

The proof of this lemma is analogous to \cite[Lemma 4.2--4.3]{ChenLeFloch}.
For the sake of completeness, we include a proof here in the relativistic setting.
It is based on the structure of the fractional derivatives of the kernels given
in Proposition \ref{prop:relativisticexplicitsingularities}.

\begin{proof}[Proof of Lemma {\rm \ref{lemma:cancellationI}(i)}]
Since we only require the fine properties of the leading order term in each of the two expansions,
we set
\beqas
\tilde{g}(\rho,v-s_1)=&\,a_2(\rho)G_{\la+1}(\rho,v-s_1)+g(\rho,v-s_1),\\
\tilde{h}(\rho,v-s_1)=&-(v-s_1)b_2(\rho)G_{\la+1}(\rho,v-s_1)+h(\rho,v-s_1).
\eeqas
With this notation, recalling that $b_1(\rho)=\rho_0(\rho)\frac{k'(\rho)}{k(\rho)}a_1(\rho)$,
we employ the expansions of Theorems \ref{thm:entropykernelmain}--\ref{thm:entropyfluxkernelmain} to write the product as
$$
\chi_1 P_j\sigma_j^\dee-P_j\chi_j^\dee \sigma_1=E^{I,\dee}+E^{II,\dee}+E^{III,\dee},
$$
where
\beqas
&E^{I,\dee}:=\,a_1\rho_0 k^{-1}k'^\half e^{\tilde{a}}G_{\la,1}\sum_{\pm} K^\pm ((s_j-s_1)\delta_{s_j=v\pm k})*\phi_j^\dee,\\
&E^{II,\dee}:=\,\rho_0k^{-1}k'^\half e^{\tilde{a}}\sum_\pm K^\pm \tilde{g}_1((s_j-s_1)\delta_{s_j=v\pm k})*\phi_j^\dee
  -k'^{-\half}e^{\tilde{a}}\sum_\pm K^\pm \tilde{h}_1\delta_{s_j=v\pm k}*\phi_j^\dee,\\
&E^{III,\dee}:=\,(a_1 G_{\la,1}+\tilde{g}_1)e^{II}_j*\phi_j^\dee-\big((s_1-v)b_1 G_{\la,1}+\tilde{h}_1\big)e^{I}_j*\phi_j^\dee.
\eeqas
Using the bounds given by Lemma \ref{lemma:k-asymptotics} for $k(\rho)$ and $k'(\rho)$, we bound $E^{I,\dee}$ by
\beqas
|E^{I,\dee}(\rho,v,s_1)|\leq&\,C\rho^{\frac{1-\th}{2}}[k-(v-s_1)]_+^\la[k+(v-s_1)]_+^\la\sum_\pm K^\pm|s_1-v\mp k|\phi_j^\dee(s_1-v\mp k)\\
\leq&\, C\rho^{\frac{1-\th}{2}}\sum_\pm K^\pm |s_1-v\mp k|^{\la+1}\dee^{-1}\phi_j(\frac{s_1-v\mp k}{\dee})\\
\leq&\,C\rho^{\frac{1-\th}{2}}\dee^\la\to0
\eeqas
locally uniformly in $(\rho,v,s_1)$, as $\dee\to0$,
where we have used the fact that $\supp\,\phi_j\subset(-1,1)$.

Next, for $E^{II,\dee}$, we make the bound{\rm :}
\beqas
|E^{II,\dee}(\rho,v-s_1)|\leq C\rho^{\frac{1-\th}{2}}[k^2-(v-s_1)^2]_+^{\la+1}\sum_\pm \phi_j^\dee(s_1-v\mp k)
\leq \,C\rho^{\frac{1-\th}{2}}\dee^\la\to0.
\eeqas
Finally, we consider the remainder term, $E^{III,\dee}(\rho,v-s_1)$.
Using bounds \eqref{ineq:eIbound}--\eqref{ineq:eIIbound}, we observe
\beqs
|s_1-v||e^I_j(\rho,v-s_1)|+|e^{II}_j(\rho,v-s_1)|
\leq C\rho^{\frac{1-\th}{2}+2\al\th}G_{-\al}(\rho,v-s_1) \qquad \mbox{ for $\al\in(0,1)$}.
\eeqs
Clearly, this is not H\"older continuous up to the boundary of its support.
However, in the region $\{|k^2-(v-s_1)^2|\leq\De\}$ for $\De>0$, we may bound
\beqs
|E^{III,\dee}(\rho,v-s_1)|\leq C\rho^{\frac{1-\th}{2}+2\al\th}G_{\la-\al}(\rho,v-s_1)\leq \rho^{3\frac{1-\th}{2}+2\al\th}\De^{\la-\al},
\eeqs
which may be made arbitrarily small by taking $0<\al<\min\{1,\la\}$ and $\De$ small.
On the other hand, in the complement region, $\{|k^2-(v-s_1)^2|>\De\}$ so that we conclude that
$G_\la$, $e_j^I$, and $e_j^{II}$ are all uniformly H\"older continuous, and hence  $E^{III,\dee}$ converges to a H\"older continuous limit
on this set, independent of the choice of mollifying sequence.
\end{proof}

The proof of Lemma \ref{lemma:cancellationI}(ii)
rests
on the observation of \cite{LionsPerthameTadmor} (and of \cite{DMLFM} in greater generality)
that the limit of a regularized product of a function of
bounded variation with a measure depends on the choice of regularization.
In particular, for the case of the product of a Heaviside function and a Dirac mass, we use the following lemma.

\begin{lemma}\label{lemma:measureheavisideproduct}
For any $m_2,m_3\in\mathbb{R}$,
$$
\left(H_{s_2=m_2}*\phi_2^\dee\right)\left(\delta_{s_3=m_3}*\phi_3^\dee\right)
\rightharpoonup\Omega^{\phi_2,\phi_3}(m_2,m_3)\delta_{s_1=m_3}
$$
weak-star in measures as $\dee\to0$, where
 \begin{equation*}
  \Omega^{\phi_2,\phi_3}(m_2,m_3):=\begin{cases}
                    0 &\text{ if $m_2>m_3$,}\\[1mm]
                    \int_{-1}^1\int_s^1\phi_2(t-s-1)\phi_3(t)\,\dd t\,\dd s &\text{ if $m_2=m_3$,}\\[1mm]
                    1 &\text{ if $m_2<m_3$}.
                   \end{cases}
 \end{equation*}
\end{lemma}

\begin{proof}[Proof of Lemma {\rm \ref{lemma:cancellationI}(ii)}]
We employ the expansions of Theorems \ref{thm:entropykernelmain} and \ref{thm:entropyfluxkernelmain} to obtain
\beqas
&P_2\chi(\rho,v-s_2) P_3\sigma(\rho,v,s_3)-P_3\chi(\rho,v-s_3) P_2\sigma(\rho,v,s_2)\\
&= P_2\chi_2 P_3(\sigma_3-\tilde{u}\chi_3)-P_3\chi_3 P_2(\sigma_2-\tilde{u}\chi_2)\\
&= (a_1P_2G_{\la,2}+a_2P_2G_{\la+1,2}+P_2g_2)\\
&\quad\,\,\times \big((s_3-v)(b_1P_3G_{\la,3}+b_2P_3G_{\la+1,3})+P_3h_3+(\la+1)b_1\partial_{s_3}^\la G_{\la,3}+(\la+1)b_2\partial_{s_3}^\la G_{\la+1,3}\big)\\
&\quad -(a_1P_3G_{\la,3}+a_2P_3G_{\la+1,3}+P_3g_3)\\
&\quad\,\,\,\,\,\, \times \big((s_2-v)(b_1P_2G_{\la,2}+b_2P_2G_{\la+1,2})
 +P_2h_2+(\la+1)b_1\partial_{s_2}^\la G_{\la,2}+(\la+1)b_2\partial_{s_2}^\la G_{\la+1,2}\big)\\
&=E^I+E^{II}+E^{III},
\eeqas
where we have decomposed the expression as
\beqas
E^{I}:=&\,(s_3-s_2)a_1b_1P_2G_{\la,2}P_3G_{\la,3},\\
E^{II}:=&\,a_1P_2G_{\la,2}\big((s_3-v)b_2P_3G_{\la+1,3}+(\la+1)b_1\partial_{s_3}^\la G_{\la,3}\big)\\
&-a_1P_3G_{\la,3}\big((s_2-v)b_2P_2G_{\la+1,2}+(\la+1)b_1\partial_{s_2}^\la G_{\la,2}\big)\\
&+a_2b_1\big(P_2G_{\la+1,2}(s_3-v)P_3G_{\la,3}-P_3G_{\la+1,3}(s_2-v)P_2G_{\la,2}\big),
\eeqas
and $E^{III}$ is the remainder.

We now take mollification for the mollifiers defined above. This yields
$$
P_2\chi_2^\dee P_3\sigma_3^\dee-P_3\chi_3^\dee P_2\sigma_2^\dee
=\big(E^I+E^{II}+E^{III}\big)*\phi_2^\dee*\phi_3^\dee.
$$
We recall that, as our mollified expressions are evaluated at $s_1$ (compare \eqref{eq:mollificationexpression}),
this is now a function of $(\rho,v,s_1)$ only.
From symmetry considerations, the limit of $E^{III,\dee}$ is $0$  as $\dee\to0$, since this term contains only the products
of measures with H\"older continuous functions and more regular products.
This convergence is uniformly in $(\rho,v)$ and weak-star in measures in $s_1$.

We consider next the most singular terms, arising in $E^{I,\dee}=E^I*\phi_2^\dee*\phi_3^\dee$.
From Proposition \ref{prop:relativisticexplicitsingularities},
we see that this expression involves products of measures, products of measures with $L^p$ functions,
and products of $L^p$ functions.
Again, by symmetry considerations, the last group of these terms vanishes in the limit as $\dee\to0$, so that we focus only on the first two.
Observe first that a typical product of measures is of the form:
$$
(s_3-s_2)a_1b_1k'(\rho)^{-1}e^{2\tilde{a}(\rho)}\delta_{s_2=v\pm k(\rho)}\,\delta_{s_3=v\pm k(\rho)}*\phi_2^\dee*\phi_3^\dee.
$$
If both Dirac masses are based at the same point $s_2=s_3=v\pm k(\rho)$, then the factor $(s_3-s_2)$ leads the expression to vanish.
Then it suffices to consider the case that they are based at different points.
The action of this measure on a continuous function $\psi(s_1)$ is then
\beqas
a_1b_1e^{2\tilde{a}}k'(\rho)^{-1}\int(w-z)\phi_2^\dee(s_1-w)\phi_3^\dee(s_1-z)\psi(s_1)\,\dd s_1.
\eeqas
Hence, we bound
\beqas
&k'(\rho)^{-1}\Big|\int(w-z)\phi_2^\dee(s_1-w)\phi_3^\dee(s_1-z)\psi(s_1)\,\dd s_1\Big|\\
&\leq C(w-z)^{1+2\la}\int\phi_2^\dee(s_1-w)\phi_3^\dee(s_1-z)|\psi(s_1)|\,\dd s_1\\
&\leq C\dee^{2\la}\big(\frac{w-z}{\dee}\big)^{1+2\la}\int_{-1}^1\phi_2(s_1)\phi_3(s_1+\frac{w-z}{\dee})|\psi(w+\dee s_1)|\,\dd s_1\\
&\leq C\dee^{2\la}\to0,
\eeqas
where we have used $\supp\,\phi_3\subset(-1,1)$.
Arguing similarly for the other terms, we obtain
$$
E^{I,\dee}\to0 \qquad \text{ as $\dee\to0$}
$$
weak-star in measures in $s_1$ and uniformly in $(\rho,v)$.

We now come to the most significant term $E^{II,\dee}=E^{II}*\phi_2^\dee*\phi_3^\dee$.
Replacing $s_2$ and $s_3$ with $s_1$ gives a remainder equipped with good factors of form $s_j-s_1$,
which may be shown to converge to $0$ as above.
Taking account of cancellations, we therefore consider the expression:
\beqas
\tilde E^{II,\dee}=&\,(\la+1)a_1b_1\big(P_2G_{\la,2}\partial_{s_3}^\la G_{\la,3}-P_3G_{\la,3}\partial_{s_2}^\la G_{\la,2}\big)*\phi_2^\dee*\phi_3^\dee\\
&+(s_1-v)(a_1b_2-a_2b_1)\big(P_2 G_{\la,2}P_3G_{\la+1,3}-P_3G_{\la,3}P_2G_{\la+1,2}\big)*\phi_2^\dee*\phi_3^\dee.
\eeqas
Observe that
$$
\partial_s^{\la+1}G_{\la+1}(\rho,v-s)=[k^2-(v-s)^2]_+\partial_s^{\la+1}G_{\la}(\rho,v-s)-2(\la+1)(s-v)\partial_s^\la G_\la(\rho,v-s),
$$
and that the contributions from the first of these terms may be seen to converge to $0$.
Thus, it suffices to consider the contribution from
\beqas
&(\la+1)a_1b_1\big(P_2G_{\la,2}\partial_{s_3}^\la G_{\la,3}-P_3G_{\la,3}\partial_{s_2}^\la G_{\la,2}\big)*\phi_2^\dee*\phi_3^\dee\\
&\quad -2(\la+1)(s_1-v)^2(a_1b_2-a_2b_1)\big(P_2 G_{\la,2}\partial_{s_3}^\la G_{\la,3}-P_3G_{\la,3}\partial_{s_2}^\la G_{\la,2}\big)*\phi_2^\dee*\phi_3^\dee\\
&=(\la+1)\big(a_1b_1-2k^2(a_1b_2-a_2b_1)\big)\big(P_2G_{\la,2}\partial_{s_3}^\la G_{\la,3}-P_3G_{\la,3}\partial_{s_2}^\la G_{\la,2}\big)*\phi_2^\dee*\phi_3^\dee
  + G_{\rm error},
\eeqas
where $G_{\rm error}$ also converges to $0$.
Applying now the expansions for the explicit singularities calculated
in Proposition \ref{prop:relativisticexplicitsingularities} and Lemma \ref{lemma:measureheavisideproduct}, we conclude the result expected.
\end{proof}

With the above results, the proof of Theorem \ref{thm:relativisticreduction}
follows the same strategy as the proof of
\cite[Theorem 4.2]{ChenLeFloch}, included here for completeness.

\begin{proof}[Proof of Theorem {\rm \ref{thm:relativisticreduction}}]
We begin by applying the commutation relation \eqref{eq:relativisticcommutation} with the weak entropy pairs
generated by test functions $\psi_1$ and $\psi_2$ and use the density of test functions to derive the commutation relation
directly for the entropy kernels themselves:
\beqs
\overline{\chi(s_1)\sigma(s_2)-\chi(s_2)\sigma(s_1)}
=\overline{\chi(s_1)}\,\overline{\sigma(s_2)}-\overline{\chi(s_2)}\,\overline{\sigma(s_1)}
\eeqs
for any $s_1,s_2\in\R$, as in the proof of Lemma \ref{lemma:relativisticvertexsupport},
where, for example, $\overline{\chi(s_1)}=\int\chi(\rho,v,s_1)\,\dd \nu(\rho,v)$.
We choose $s_1,s_2,s_3\in\R$ and apply this identity to each of the pairs $(s_2,s_3)$, $(s_3,s_1)$, and $(s_1,s_2)$.
Multiplying these identities by $\overline{\chi(s_1)}$, $\overline{\chi(s_2)}$, and $\overline{\chi(s_3)}$, respectively,
and summing them together,
we see that the right-hand side vanishes (by an obvious symmetry), which leads to
\beqas
&\overline{\chi(s_1)}\,\,\overline{\chi(s_2)\sigma(s_3)-\chi(s_3)\sigma(s_2)}
+\overline{\chi(s_2)}\,\,\overline{\chi(s_3)\sigma(s_1)-\chi(s_1)\sigma(s_3)}\\[1mm]
& +\overline{\chi(s_3)}\,\,\overline{\chi(s_1)\sigma(s_2)-\chi(s_2)\sigma(s_1)}=0.
\eeqas
Applying now operators $P_2$ and $P_3$ defined above, we obtain
\beqas
&\overline{\chi(s_1)}\,\,\overline{P_2\chi(s_2)P_3\sigma(s_3)-P_3\chi(s_3)P_2\sigma(s_2)}
 +\overline{P_2\chi(s_2)}\,\,\overline{P_3\chi(s_3)\sigma(s_1)-\chi(s_1)P_3\sigma(s_3)}\\[1mm]
&+\overline{P_3\chi(s_3)}\,\,\overline{\chi(s_1)P_2\sigma(s_2)-P_2\chi(s_2)\sigma(s_1)}=0
\eeqas
distributionally in $(s_1,s_2,s_3)$.

We mollify this expression with mollifiers $\phi_2$ and $\phi_3$ as described above to obtain
\beqa\label{eq:capybara}
\overline{\chi_1}\,\,\overline{P_2\chi^\dee_2P_3\sigma^\dee_3-P_3\chi^\dee_3P_2\sigma^\dee_2}
 +\overline{P_2\chi^\dee_2}\,\,\overline{P_3\chi^\dee_3\sigma_1-\chi_1P_3\sigma^\dee_3}
+\overline{P_3\chi^\dee_3}\,\,\overline{\chi_1 P_2\sigma^\dee_2-P_2\chi^\dee_2\sigma_1}=0
\eeqa
with obvious notation.
Passing now $\dee\to0$, we recall that, as $P_j\chi_j$ is a bounded measure in $s_j$ with coefficients uniformly bounded in $(\rho,v)$,
we may pass
$$
P_j\chi_j^\dee\to P_1\chi_1
$$
weak-star with respect to measures in $s_1$ and uniformly with respect to $(\rho,v)$. Therefore, we have
$$
\overline{P_j\chi^\dee_j}\to\overline{P_1\chi_1}
$$
weak-star in measures in $s_1$.

Considering now the last two terms of \eqref{eq:capybara},
we may combine this convergence with the uniform convergence of Lemma \ref{lemma:cancellationI}(i) to deduce
$$
\overline{P_2\chi^\dee_2}\,\,\overline{P_3\chi^\dee_3\sigma_1-\chi_1P_3\sigma^\dee_3}+\overline{P_3\chi^\dee_3}\,\,\overline{\chi_1 P_2\sigma^\dee_2-P_2\chi^\dee_2\sigma_1}\to\overline{P_1\chi_1}\,\,\overline{X_1}-\overline{P_1\chi_1}\,\,\overline{X_1}=0,
$$
weak-star in measures in $s_1$.
On the other hand, we may apply Lemma \ref{lemma:cancellationI}(ii) to deduce that the first term of \eqref{eq:capybara}
converges weak-star in measures in $s_1$ to
$$
\overline{\chi(s_1)}\,\,\overline{Y(\phi_2,\phi_3)M(\rho)D(\rho)\sum_\pm (K^\pm)^2\delta_{s_1=v\pm k(\rho)}},
$$
so that,  for any test function $\psi(s_1)$,
$$
Y(\phi_2,\phi_3)\sum_\pm(K^\pm)^2\int\overline{\chi(v\pm k(\rho))}M(\rho)D(\rho)\psi(v\pm k(\rho))\,\dd \nu(\rho,v)=0.
$$
By assumption, we take $Y(\phi_2,\phi_3)\neq 0$ so that
$$
\sum_\pm\int\overline{\chi(v\pm k(\rho))}M(\rho)D(\rho)\,\dd \nu(\rho,v)=0.
$$
As $M(\rho)D(\rho)>0$ for $\rho>0$ by Proposition \ref{prop:Dcoefficient} (recall that $\eps<\eps_0$),
we deduce
$$
\supp\,\nu\cap\{(\rho,u)\,:\,z_{\min}< z(\rho,u)<w(\rho,u)<w_{\max}\}=\emptyset,
$$
since, for all $s\in(z_{\min},w_{\max})$, $\chi(s)$ (considered in the $(w,z)$--coordinates) contains
point $(w_{\max},z_{\min})$ in the interior of its support and, by Lemma \ref{lemma:relativisticvertexsupport},
point $(w_{\max},z_{\min})\in\supp\,\nu$ so that
 $\overline{\chi(s)}>0$ for all $s\in(z_{\min},w_{\max})$.
Thus, the support of $\nu$ must be contained in the vacuum  line $V$ and point $(w_{\max},z_{\min})$.
Writing
$$
\nu=\nu_V+\om\de_{(w_{\max},z_{\min})},
$$
where $\nu_V$ is supported in the vacuum state $V$ and $\om\in[0,1]$,
we deduce from the commutation relation that, for all $s_1,s_2\in\R$,
$$
(\om-\om^2)\big(\chi(w_{\max},z_{\min},s_1)\sigma(w_{\max},z_{\min},s_2)-\chi(w_{\max},z_{\min},s_2)\sigma(w_{\max},z_{\min},s_1)\big)=0.
$$
Choosing $s_1$ and $s_2$ such that the second factor is non-zero, we deduce that $\om=0$ or $\om=1$.
\end{proof}

Then the proof of Theorem \ref{thm:relativisticcompactness} directly follows from Theorem \ref{thm:relativisticreduction}.

\section{Global Viscosity Solutions}\label{sec:relativisticartificial}
In this section, we demonstrate a method for the construction of a sequence of approximate solutions
satisfying the compactness framework above.
We address this problem via the introduction of \it artificial viscosity \rm by considering the system:
\beq\label{eq:relativisticartificialviscosity}
\begin{cases}
\partial_t\big(\frac{n}{\sqrt{1-u^2/c^2}}\big) + \partial_x\big(\frac{nu}{\sqrt{1-u^2/c^2}}\big) = \de\partial_{xx}\big(\frac{n}{\sqrt{1-u^2/c^2}}\big), \\[3mm]
\partial_t\big(\frac{(\rho + p/c^2)u}{1-u^2/c^2}\big) + \partial_x\big(\frac{(\rho + p/c^2)u^2}{1-u^2/c^2} + p\big) = \de\partial_{xx}\big(\frac{(\rho + p/c^2)u}{1-u^2/c^2}\big),
\end{cases}
\eeq
where $\de>0$ is the viscosity parameter.
As with the classical Euler equations, the viscosity system admits an \it invariant region\rm,
which is one of the conditions we require to apply our compactness framework.

Before we state the theorem for the existence of the solutions to this system,
a few remarks on the end-point states are in order.
To allow for the possibility that the density and velocity do not vanish at infinity,
we impose the end-point states $(\rho_\pm,u_\pm)$ such that $\rho_\pm\geq 0$ and $|u_\pm|< c$ for the approximate solutions.
We introduce smooth, monotone functions $(\bar\rho(x),\bar u(x))$ such that
$(\bar\rho(x),\bar u(x))=(\rho_\pm,u_\pm)$ for $\pm x\ge 1$ and require
that the approximate initial data functions satisfy
$(\rho^\de_0-\bar\rho,u^\de_0-\bar u)\in C_{\rm c}^\infty(\R)$.
The existence and uniform bounds of solutions for this system are given in the following theorem.

\begin{thm}\label{thm:relativisticapproxsolutions}
Let $(\rho_0^\de, u_0^\de)$ be approximate initial data functions such that
\beqs
(\rho_0^\de-\bar\rho,u_0^\de-\bar u)\in C_{\rm c}^\infty(\R)
\eeqs
and, for some $M_0>0$ independent of $\de>0$,
\beqs
|u_0^\de|\leq M_0<\frac{1}{\sqrt{\eps}}, \quad\,\, 0<\rho_0^\de\leq \rho_{M_0}<\rho_{\max}^\eps.
\eeqs
Then there exist global solutions $U^\de=U(\rho^\de,u^\de)$ of system \eqref{eq:relativisticartificialviscosity} such that
$$
(\rho^\de(t,\cdot)-\bar\rho,u^\de(t,\cdot)-\bar u)\in C^1\cap H^1
$$
and
$$
|u^\de(t,x)|\leq M<\frac{1}{\sqrt{\eps}}, \qquad 0<\rho^\de(t,x)\leq \rho_M<\rho_{\max}^\eps \,\qquad \text{ for all }(t,x)\in\R^2_+,$$
where $M$ and $\rho_M$ are independent of $\de>0$.
\end{thm}

The proof is by now standard.
The uniform bounds on $\rho$ and $u$ follow from the following lemma,
whose proof is a standard argument based on the parabolic maximum principle for the Riemann invariants.
Throughout this section, we drop the explicit dependence of the functions on $\de>0$, which is assumed to be fixed.

\begin{lemma}\label{lemma:relativisticapproxinvariantregion}
Any $C^{1,1}$ solution $(\rho(t,x),u(t,x))$ to system \eqref{eq:relativisticartificialviscosity} admits the following bounds{\rm :}
\beqs
\|(k(\rho),v(u))\|_{L^\infty(\R^2_+)}\leq C\|(w_0,z_0)\|_{L^\infty(\R)},
\eeqs
where $w_0(x)=w(\rho_0(x),u_0(x))$, $z_0(x)=z(\rho_0(x),u_0(x))$, and $C>0$ is independent of $\delta$ and $\eps$,
which implies that there exist $M$ and $\rho_M$ depending on $\|(w_0,z_0)\|_{L^\infty(\R)}$,
but independent of $\eps$, such that
\begin{align*}
|u(t,x)|\leq M<\frac{1}{\sqrt{\eps}}, \quad\, 0\le \rho\le \rho_M  <\rho_{\max}^\eps.
\end{align*}
\end{lemma}

To apply the compactness framework in \S \ref{sec:relativisticcompactness},
we require the $H^{-1}_{\loc}$ compactness of the entropy dissipation measures
$\partial_t\eta(U^\de)+\partial_x q(U^\de)$.
This compactness requires the uniform estimates  of $(\rh^\de_x, u^\de_x)$ in $\de>0$.
We obtain this via the relative entropy method.

We recall the physical entropy pair $(\eta^*,q^*)$ from \eqref{def:physicalentropy}.
Writing $U=U(\rho,u)$ and $\bar U=U(\bar\rho,\bar u)$,
we define a modified entropy pair (designed via the relative entropy method) by
\beq\label{def:relativisticmodifiedentropy}
\overline{\eta^*}(U)=\eta^*(U)-\eta^*(\bar U)-\nabla\eta^*(\bar U)\cdot(U-\bar U)\geq 0.
\eeq
As $\eta^*$ is convex, we see that $\overline{\eta^*}(U)\geq0$ and $\nabla^2\overline{\eta^*}(U)=\nabla^2\eta^*(U)$.

\begin{lemma}\label{lemma:relativisticderivativebounds}
Suppose that
$$
\int_\R\overline{\eta^*}(U(\rho_0,u_0))(x)\,\dd x<\infty.
$$
Then there exists $C>0$, independent of $\de$, such that any solution $U(\rho,u)$ of \eqref{eq:relativisticartificialviscosity} satisfies
\beqs
\sup_{t\in[0,T]}\int_\R\overline{\eta^*}(U(\rho,u))(t,x)\,\dd x+\de\int_0^T\int_\R(U_x)^\top\nabla^2\eta^{*}(U) U_x\,\dd x\,\dd t\leq C.
\eeqs
In particular,
\beq\label{ineq:energy}
\de\int_0^T\int_\R\big(\rho^{\ga-2}|\rho_x|^2+\rho|u_x|^2\big)\,\dd x\,\dd t\leq C.
\eeq
\end{lemma}

The proof is by now standard.
When $\ga\leq 2$, it suffices to deduce the desired compactness of the entropy dissipation measures.
However, when $\ga>2$, this bound is insufficient. The following lemma provides the necessary improved control.

\begin{lemma}\label{lemma:relativisticga>2rhoxbound}
Let $\De\in(0,\half)$, and let $K\subset\R$ be compact.
Then any solution $(\rho,u)$ of \eqref{eq:relativisticartificialviscosity} satisfies
$$
\de\int_0^T\int_{K\cap\{\rho<\De\}}|\rho_x|^2\,\dd x\,\dd t\leq C\De+C\frac{\De^2}{\de}+C\De^{\frac{4-\ga}{2}}
$$
for some $C>0$ independent of $\de>0$.
\end{lemma}

\begin{proof}
We denote $N=\frac{n(\rho)}{\sqrt{1-\eps u^2}}$ so that the first equation in \eqref{eq:relativisticartificialviscosity} becomes
$$
N_t+(Nu)_x=\de N_{xx}.
$$
For $\De\in(0,\half)$ to be determined later, we set
\beqs
\phi(N):=\begin{cases}
\half N^2 &\mbox{for $N<\De$},\\[1mm]
\half \De^2+\De(N-\De) &\mbox{for $N\geq\De$}.
\end{cases}
\eeqs
Observe now that $\phi'(N)=N\mathds{1}_{N<\De}+\De\mathds{1}_{N\geq\De}=\min\{N,\De\}$ and $\phi''(N)=\mathds{1}_{N<\De}$.
Multiplying the first equation in \eqref{eq:relativisticartificialviscosity} by $\phi'(N)\om^2(x)$,
where $\om\in C_{\rm c}^\infty(\R)$ is a spatial test function such that $\om=1$ on $K$ and $\om\geq 0$,
$$
(\phi(N)\om^2)_t+\phi'(N)(Nu)_x\om^2 =\de\phi'(N)N_{xx}\om^2.
$$
Then integrating by parts yields
\beqas
&\int_{\supp\,\om}\phi(N)\om^2\,\dd x\big|_0^T-\int_0^T\int_{\supp\,\om}\phi''(N)N_x Nu\om^2\,\dd x\,\dd t-\int_0^T\int_{\supp\,\om}2\phi'(N)Nu\om_x\om\,\dd x\,\dd t\\
&=-\de\int_0^T\int_{\supp\,\om}\phi''(N)N_x^2\om^2\,\dd x\,\dd t-2\de\int_0^T\int_{\supp\,\om}\phi'(N)N_x\om_x\om\,\dd x\,\dd t.
\eeqas
Rearranging this equation and recalling the expressions for $\phi'(N)$ and $\phi''(N)$ above, we see
\beqas
&\de\int_0^T\int_{{\supp\,\om}\cap\{N<\De\}}N_x^2\om^2\,\dd x\,\dd t\\
&=\int_{\supp\,\om}\phi(N)\om^2\,\dd x\big|_0^T+\int_0^T\int_{{\supp\,\om}\cap\{N\leq\De\}}N_xNu\om^2\,\dd x\,\dd t\\
&\quad +2\int_0^T\int_{\supp\,\om}\min\{N,\De\}Nu\om_x\om\,\dd x\,\dd t-2\de\int_0^T\int_{{\supp\,\om}\cap\{N<\De\}}NN_x\om_x\om\,\dd x\,\dd t\\
&\quad -2\de\int_0^T\int_{\supp\,\om}\De N_x\om_x\om\,\dd x\,\dd t.
\eeqas
Recall now that there is a uniform bound on $(N, u)$ from Lemma \ref{lemma:relativisticapproxinvariantregion},
as well as on $(\om,\om_x)$. We therefore note that $\phi(N)\leq\De$ and apply the H\"older inequality to obtain
\beqas
\de\int_0^T&\int_{{\supp\,\om}\cap\{N<\De\}}N_x^2\om^2\,\dd x\,\dd t
\leq C\De +C\De\Big(\int_0^T\int_{{\supp\,\om}\cap\{N<\De\}}N_x^2\om^2\,\dd x\,\dd t\Big)^{\half}+C\sqrt{\de}\De^{\frac{4-\ga}{2}},
\eeqas
where we have used
\beqs
\de\int_0^T\int_{\supp\,\om}\mathds{1}_{N\geq\De}\De^2N_x^2\om^2\,\dd x\,\dd t\leq C\De^{4-\ga}.
\eeqs
Thus
\beqs
\de\int_0^T\int_{K\cap\{N<\De\}}N_x^2\om^2\,\dd x\,\dd t\leq C\Big(\De+\frac{\De^2}{\de}+\De^{\frac{4-\ga}{2}}\Big).
\eeqs
Note that $N_x=\frac{\rho_x n(\rho)}{(\rho+\eps p)\sqrt{1-\eps u^2}}+\frac{\eps nuu_x}{(1-\eps u^2)^{3/2}}$. Then
$$
N_x^2\geq \half \rho_x^2-C\frac{(\eps nuu_x)^2}{(1-\eps u^2)^{3}}\geq \half\rho_x^2-CN\rho u_x^2,
$$
where we have used that $n'(\rho)=\frac{n}{\rho+\eps p}$, the uniform bound on $u$, and $n(\rho)\leq C\rho$.
Thus, applying again the energy estimate of Lemma \ref{lemma:relativisticderivativebounds}, we have
\beqs
\de\int_0^T\int_{K\cap\{N<\De\}}\rho_x^2\om^2\,\dd x\,\dd t\leq C\Big(\De+\frac{\De^2}{\de}+\De^{\frac{4-\ga}{2}}\Big),
\eeqs
which is the desired conclusion since there exists a constant $C>0$ such that $C^{-1}\rho\leq N\leq C\rho$.
\end{proof}

We also use the following fact, verified by direct calculation from the representation formula of Theorem \ref{thm:entropykernelmain}.

\begin{lemma}\label{lemma:relativisticentropyhessian}
For any weak entropy pair $(\eta,q)$, the ordering
 $$|\nabla^2\eta|\leq C\nabla^2\eta_*$$
 holds with matrices ordered in the usual way.
\end{lemma}

We prove the $H^{-1}_{\loc}$ compactness of the entropy dissipation measures in the following lemma.

\begin{prop}
Let $U^\de=U(\rho^\de,u^\de)$ be a sequence of solutions of \eqref{eq:relativisticartificialviscosity} with initial data $U^\de_0$
satisfying the assumptions of Theorem {\rm \ref{thm:relativisticapproxsolutions}} and Lemma {\rm \ref{lemma:relativisticderivativebounds}}.
Then, for any weak entropy pair $(\eta,q)$, the sequence of entropy dissipation measures
\beq
\eta(U^\de)_t+q(U^\de)_x \qquad \text{ is compact in $H^{-1}_{\loc}(\R^2_+)$.}
\eeq
\end{prop}

\begin{proof}
Throughout the proof, we write $(\eta^\de, q^\de):=(\eta(U^\de), q(U^\de))$.
Multiplying the equation:
$$
U^\de_t+F(U^\de)_x=\de U^\de_{xx}
$$
by $\nabla\eta(U^\de)$, we find
\beq\label{eq:approximateentropyform}
\eta^\de_t+q^\de_x=\de\eta^\de_{xx}-\de (U^\de_x)^\top\nabla^2\eta^\de U^\de_x.
\eeq
By Lemmas \ref{ineq:energy} and \ref{lemma:relativisticentropyhessian},
we may bound the last term in $L^1_{\loc}(\R^2_+)$ by observing that, for any compact $K\subset\R$,
\beqs
\int_0^T\int_K\big|\de (U^\de_x)^\top\nabla^2\eta^\de U^\de_x\big|\,\dd x\,\dd t\leq\int_0^T\int_K\big|\de (U^\de_x)^\top\nabla^2\eta_*^{\de} U^\de_x\big|\,\dd x\,\dd t\leq C,
\eeqs
independent of $\de>0$.
Applying the compact embedding of $L^1(K)$ into $W^{-1,q}(K)$ with $1<q<2$, we find that
this term is compact in $W^{-1,q}_{\loc}$ for $q<2$.

For the first term $\de\eta^\de_{xx}$ on the right-hand side of \eqref{eq:approximateentropyform}, we note that
$$
|\eta_x^\de|\leq C\big(|\rho_x^\de|(1+(\rho^{\de})^\th)+\rho^\de|u_x^\de|\big).
$$
Now, for $1<\ga\leq 2$, Lemma \ref{lemma:relativisticderivativebounds} implies that
$$
\sqrt{\de}\eta^\de_x \qquad \text{ is uniformly bounded in $L^2(\R_+^2)$,}
$$
so that
$$
\de \eta^\de_{xx} \qquad \text{ is compact in $W^{-1,2}_{\loc}(\R_+^2)$.}
$$
In the case that $\ga>2$, we apply the estimate of Lemma \ref{lemma:relativisticga>2rhoxbound} to deduce that,
on a compact set $K\subset\R$,
\beqas
\de^2\int_0^T\int_K|\eta_x^\de|^2\,\dd x\,\dd t
\leq&\, C\de^2\int_0^T\int_K\big(|\rho_x^\de|^2+\rho^\de|u_x^\de|^2\big)\,\dd x\,\dd t\\
\leq&\,C\de^2\int_0^T\int_K\big(|\rho_x^\de|^2(\mathds{1}_{\rho^\de<\De}
  +\mathds{1}_{\rho^\de\geq\De})+\rho^\de|u_x^\de|^2\big)\,\dd x\,\dd t\\
\leq&\, C\big(\de\De+\De^2+\de\De^{\frac{4-\ga}{2}}+\de\De^{2-\ga}+\de),
\eeqas
where $\De>0$ is to be chosen now.
In fact, choosing $\De=\de^\al$ with $\al\in(0,\frac{1}{\ga-2})$ implies that
this expression converges to $0$. Hence, we see that
$$
\de\eta_{xx}^\de \qquad \text{ is compact in $W^{-1,2}_{\loc}(\R_+^2)$.}
$$
We have therefore shown that the sequence of entropy dissipation measures
$$
\eta^\de_t+q^\de_x \qquad \text{ is compact in $W^{-1,q}_{\loc}(\R_+^2)$}\qquad \mbox{ for some $q\in(1,2)$}.
$$
On the other hand, since the approximate solutions $(\rho^\de,u^\de)$ are uniformly bounded, we also have
$$
\eta^\de_t+q^\de_x \qquad \text{ is bounded uniformly in $W^{-1,\infty}_{\loc}(\R_+^2)$.}
$$
Applying now the compensated compactness interpolation theorem (\textit{cf}.~\cite{Chen2}), we see that
$$
\eta^\de_t+q^\de_x \qquad \text{ is compact in $W^{-1,2}_{\loc}(\R_+^2)$,}
$$
as desired.
\end{proof}

It is an easy exercise to see that the same arguments apply also to the vanishing viscosity approximation
to the alternative system \eqref{eq:relativistic-Euler-alternative}.

\section{Proof of the Main Theorem}\label{sec:mainresultproof}
For clarity, we restate the main theorem of this paper here.

\begin{thm}[Existence of Entropy Solutions]
 Let $(\rho_0,u_0)$ be measurable and bounded initial
 data function satisfying
$$
  |u_0(x)|\leq M_0<\frac{1}{\sqrt{\eps}},\quad   0\leq\rho_0(x)\leq \rho_{M_0}<\rho_{\max}^\eps \qquad\,\, \text{ for a.e. $x\in\R$,}
$$
for some $M_0$ and $\rho_{M_0}$, and let the pressure  function $p(\rho)$ satisfy $p'(\rho)>0$, \eqref{ass:genuine-nonlinearity}, and \eqref{ass:pressure}.
Then there exists $\eps_0>0$ such that the following holds{\rm :}
There exists a sequence $(\rho^\de,u^\de)$ of the viscosity solutions to the approximate equations \eqref{eq:relativisticartificialviscosity} such that,
if $\eps\leq\eps_0$,
then sequence $(\rho^\de,u^\de)$ converges to an entropy solution $(\rho,u)$ of \eqref{eq:relativistic-Euler}
in the sense of Definition {\rm \ref{def:relativistic-entropy-solution}} such that
$$
|u(t,x)|\leq M<\frac{1}{\sqrt{\eps}},\quad 0\leq\rho(t,x)\leq \rho_M<\rho_{\max}^\eps \qquad \,\, \text{ for a.e. $(t,x)\in\R^2_+$},
$$
for some $M$ and $\rho_M$ depending only on $M_0$ and $\rho_{M_0}$,
where the convergence is a.e. and strong in $L^p_{\loc}(\R^2_+)$ for all $p\in[1,\infty)$.
\end{thm}

To prove the theorem, we begin by constructing approximate initial data satisfying the assumptions
of Theorem \ref{thm:relativisticapproxsolutions} and Lemma \ref{lemma:relativisticderivativebounds}.
For each $\de>0$, we cut off the initial data outside interval $(-\de^{-1},\de^{-1})$, mollify with a standard Friedrichs mollifier,
and add a small positive constant $\bar\rho^\de>0$ to the approximate density.
This gives us the initial data $(\rho_0^\de,u_0^\de)$ satisfying the desired assumptions such that
$(\rho_0^\de,u_0^\de)\to(\rho_0,u_0)$ {\it a.e.} and in $L^p_{\loc}(\R)$.
Then Theorem \ref{thm:relativisticapproxsolutions} gives the existence of the approximate solutions,
and the results of \S \ref{sec:relativisticartificial} then imply that these approximate solutions
are uniformly bounded and satisfy the condition:
$$
\eta(U^\de)_t+q(U^\de)_x \qquad \text{ are compact in $H^{-1}_{\loc}(\R^2_+)$}
$$
for all weak entropy pairs $(\eta, q)$.

The uniform bound on the approximate solutions gives a constant $\rho_M>0$ independent of $\de$
such that $\rho^\de\leq\rho_M$ for all $\de$.
Taking $\eps_0>0$ as in Proposition \ref{prop:Dcoefficient} then implies that
the sequence of approximate solutions $(\rho^\de,u^\de)$ satisfies the assumptions
of the compactness framework of \S\ref{sec:relativisticcompactness}.
Thus Theorem \ref{thm:relativisticcompactness} gives the strong convergence of the approximate
solutions $(\rho^\de,u^\de)\to(\rho,u)$ in $L^r_{\loc}(\R^2_+)$ for all $r\in[1,\infty)$.
We can see that the obtained limit is an entropy solution of \eqref{eq:relativistic-Euler}.

Finally, we remark that the proof of the existence part of Theorem \ref{thm:alternativesystem} is similar.

\section{Newtonian Limit}\label{sec:Newtonianlimit}
The final section of this paper is devoted to
the proof of our second main result, Theorem \ref{thm:Newtonianlimit}, concerning the Newtonian limit.
We first introduce some notation to make the explicit dependence of the quantities on $\eps$ in this section.
For $\eps\in(0,1)$, we rewrite \eqref{eq:relativistic-Euler} as
\begin{equation}\label{eq:relativistic-Eulerexplicitdependence}
 \partial_t U^\eps + \partial_xF^\eps(U^\eps)=0,
\end{equation}
where $U^\eps(\rho^\eps,u^\eps)= \Big(\frac{n(\rho^\eps)}{\sqrt{1 - \eps (u^\eps)^2}},\frac{(\rho^\eps + \eps p(\rho^\eps))u^\eps}{1 - \eps (u^\eps)^2}\Big)^\top$
and $F^\eps$ is the associated flux. We recall the definition of $v^\eps=v^\eps(u^\eps)$ from \S \ref{sec:relativisticproperties}:
$$
v^\eps(u^\eps)=\frac{1}{2\sqrt{\eps}}\log\big(\frac{1+\sqrt{\eps}u^\eps}{1-\sqrt{\eps}u^\eps}\big).
$$
We write the Euler equations \eqref{eq:classical-Euler} as
\begin{equation}\label{eq:classical-Eulerexplicitdependence}
 \partial_t U + \partial_x F(U)=0,
\end{equation}
where $U=(\rho,m)^\top$, $F(U)$ is the associated flux, and $F=(m,\frac{m^2}{\rho}+p(\rho))^\top$.

\subsection{Proof of Theorem \ref{thm:Newtonianlimit}}
By assumption,  $(\rho^\eps,v^\eps)$ is a uniformly bounded sequence in $L^\infty$.
Then we may extract a subsequence such that
$$
(\rho^\eps,v^\eps)\stackrel{*}{\rightharpoonup}(\rho,v) \qquad \mbox{weakly-star in $L^\infty$}.
$$
To this sequence, we associate a Young measure $\nu_{t,x}$.
Our aim now is to show that we may apply a reduction argument analogous to that
of Theorem \ref{thm:relativisticreduction}.
However, we observe that Theorem \ref{thm:relativisticreduction} holds only for fixed $\eps>0$,
not for a sequence. We therefore recall the following theorem from \cite{ChenLeFloch, ChenLeFloch2}.

\begin{thm}[Theorem 4.2, \cite{ChenLeFloch}]\label{thm:classicalreduction}
Let $\nu(\rho,v)$ be a probability measure with bounded support in $\{\rho\geq0,\,v\in\R\}$ such that
$$
\langle \nu,\chi_*(s_1) \sigma_*(s_2)-\chi_*(s_2) \sigma_*(s_1)\rangle
=\langle\nu,\chi_*(s_1)\rangle\langle\nu,\sigma_*(s_2)\rangle-\langle\nu,\chi_*(s_2)\rangle\langle\nu,\sigma_*(s_1)\rangle
$$
for any $s_1,s_2\in\R$, where $\chi_*$ and $\sigma_*$ are the entropy and entropy-flux kernels of the classical Euler equations \eqref{eq:classical-Euler}.
Then the support of $\nu$ is either a single point or a subset of the vacuum line, $\{\rho=0\}$.
\end{thm}

Defining $m^\eps:=\rho^\eps v^\eps$, we show that
$$
\eta(\rho^\eps,m^\eps)_t+ q(\rho^\eps,m^\eps)_x\qquad \text{ is compact in $H^{-1}_{\loc}$}
$$
for any weak entropy pair of $(\eta,q)$ of system \eqref{eq:classical-Eulerexplicitdependence},
and hence argue by using the div-curl lemma as in the proof of Theorem \ref{thm:relativisticcompactness}
to deduce the commutation relation:
$$
\langle \nu,\chi_*(s_1) \sigma_*(s_2)-\chi_*(s_2) \sigma_*(s_1)\rangle
=\langle\nu,\chi_*(s_1)\rangle\langle\nu,\sigma_*(s_2)\rangle-\langle\nu,\chi_*(s_2)\rangle\langle\nu,\sigma_*(s_1)\rangle.
$$
We may then apply Theorem \ref{thm:classicalreduction} to deduce the strong convergence
of the sequence $(\rho^\eps,v^\eps)$.

It remains to prove the $H^{-1}$ compactness of the entropy dissipation measures.
In a slight abuse of notation, we simply write the argument of $(\eta,q)$ as $U^\eps$.
Then
\begin{equation}\label{entdiss}
\eta(U^\eps)_t+q(U^\eps)_x
=\big(\eta(U^\eps)-\eta^\eps(U^\eps)\big)_t
 +\big(q(U^\eps)-q^\eps(U^\eps)\big)_x+\eta^\eps(U^\eps)_t+q^\eps(U^\eps)_x,
 \end{equation}
where $(\eta, q)$ and $(\eta^\eps, q^\eps)$ are weak entropy pairs generated by the same test function $\psi$ by convolution with the
weak entropy and entropy-flux kernels associated to systems \eqref{eq:classical-Eulerexplicitdependence}
and \eqref{eq:relativistic-Eulerexplicitdependence},
respectively.

Since, for each $\eps$, $U^\eps$ is an entropy solution of \eqref{eq:relativistic-Eulerexplicitdependence}, we know that,
for any convex weak entropy pairs $(\eta^\eps, q^\eps)$,
$$
\eta^\eps(U^\eps)_t+q^\eps(U^\eps)_x\leq 0.
$$
Moreover, $\eta^\eps(U^\eps)_t+q^\eps(U^\eps)_x$ is uniformly bounded in $H^{-1}_{\loc}$ since $U^\eps$ is uniformly bounded.
Then Murat's lemma \cite{Murat2} indicates that the injection of the positive cone in $H^{-1}$ into $W^{-1, q}, 1\le q<2$, is compact, which implies
that the convex entropy dissipation measure sequence $\eta^\eps(U^\eps)_t+q^\eps(U^\eps)_x$
is  compact in $W^{-1,q}_{\loc}$ for any $q\in [1, 2)$.
Likewise, for any concave entropy pair $(\eta, q)$, we have
$$
\eta^\eps(U^\eps)_t+q^\eps(U^\eps)_x \qquad \text{ is compact in $W^{-1,q}_{\loc}$ for any $q<2$.}
$$
On the other hand, this expression is clearly bounded in  $W^{-1,\infty}_{\loc}$ by the uniform bounds
on $U^\eps$.
Then the interpolation compactness theorem of \cite{Chen2,DingChenLuo} yields that
$$
\eta^\eps(U^\eps)_t+q^\eps(U^\eps)_x \qquad \text{ is compact in $W^{-1,2}_{\loc}$}
$$
for any convex or concave entropy.

Considering now the first term on the right-hand side of \eqref{entdiss},
$$
\big(\eta(U^\eps)-\eta^\eps(U^\eps)\big)_t+\big(q(U^\eps)-q^\eps(U^\eps)\big)_x,
$$
we apply the bounds of Theorems \ref{thm:chipositive} and \ref{th:fluxremainderexistenceandregularity}
to estimate the difference between the relativistic and classical entropies generated by the same test function $\psi(s)$ by
\begin{equation}
 \begin{split}
  |\eta(\rho,v)-\eta^\eps(\rho,v)|\leq&\,\int_{|v-s|\leq k(\rho)}|\psi(s)||\chi_*(\rho,v-s)-\chi^\eps(\rho,v-s)|\,\dd s\\
  \leq& \,C\eps\int_{|v-s|\leq k(\rho)}|\psi(s)|\,\dd s,
 \end{split}
\end{equation}
where $\chi^\eps$ is the relativistic entropy kernel of Theorem \ref{thm:entropykernelmain}. Similarly, $|q(\rho,v)-q^\eps(\rho,v)|\leq C\eps$.

Thus, for any set $K\Subset\mathbb{R}^2_+$, we have
\begin{equation}
  (\eta(U^\eps)-\eta^\eps(U^\eps))_t+(q(U^\eps)-q^\eps(U^\eps))_x\to0 \qquad \text{ in $H^{-1}(K)$ as $\eps\rightarrow0$.}
\end{equation}
Therefore, the sequence of weak entropy dissipation measures in \eqref{entdiss} is compact in $H^{-1}_{\loc}$.
We have therefore shown the compactness of the weak entropy dissipation measures \eqref{entdiss}
such that
the entropy functions are either convex or concave.
Applying the div-curl lemma in the standard way, as in the proof of Theorem \ref{thm:relativisticcompactness},
we deduce the commutation relation
$$
\langle \nu,\eta_1 q_2-\eta_2 q_1\rangle=\langle\nu,\eta_1\rangle\langle\nu,q_2\rangle-\langle\nu,\eta_2\rangle\langle\nu,q_1\rangle
$$
for all such weak entropy pairs.

We recall now that a weak entropy for the classical Euler equations \eqref{eq:classical-Eulerexplicitdependence}
is convex (respectively concave) if the generating test function is also convex (respectively concave).
As the linear span of convex functions and concave functions is dense in our space of test functions,
we argue by density
to conclude that the commutation relation also holds for the kernels themselves:
$$
\langle \nu,\chi_*(s_1) \sigma_*(s_2)-\chi_*(s_2) \sigma_*(s_1)\rangle
=\langle\nu,\chi_*(s_1)\rangle\langle\nu,\sigma_*(s_2)\rangle-\langle\nu,\chi_*(s_2)\rangle\langle\nu,\sigma_*(s_1)\rangle
$$
for any $s_1,s_2\in\R$. We therefore apply Theorem \ref{thm:classicalreduction} to deduce
the strong convergence of sequence $(\rho^\eps,v^\eps)$ {\it a.e.}
and in $L^r_{\loc}(\R^2_+)$ for all $r\in[1,\infty)$.
Then the proof of Theorem \ref{thm:Newtonianlimit} is concluded in the usual way.

The proof of the Newtonian limit of Theorem 1.4(ii) is similar.

\bigskip
{\bf Acknowledgements}.  The research of Gui-Qiang G. Chen was supported in part by
the UK Engineering and Physical Sciences Research Council Award
EP/L015811/1, EP/V008854, and EP/V051121/1, and the Royal Society--Wolfson Research Merit Award WM090014 (UK).
The research of Matthew Schrecker was supported in part by the UK Engineering and Physical Sciences Research
Council Award EP/L015811/1.
This paper is a continuation of the program initiated by Gui-Qiang G. Chen, Philippe LeFloch,
and Yachun Li in \cite{ChenLeFloch, ChenLeFloch2, ChenLi, ChenLi2}.
The authors thank Professors Ph. LeFloch and Yachun Li for their inputs and helpful discussions.

\end{document}